\documentclass[10pt,a4paper]{amsart}

\usepackage{amsmath,amsthm,amssymb}

\usepackage{url}
\usepackage{amssymb,wasysym,enumerate}
\usepackage[colorlinks]{hyperref}

\usepackage[utf8]{inputenc}
\usepackage{mathpazo}

\author{{\L}ukasz Garncarek}

\title{Boundary representations of hyperbolic groups}


\address{Institute of Mathematics, Polish Academy of Sciences, ul. {\'S}niadeckich 8, 00-656 Warszawa, Poland}

\email{lukgar@impan.pl}

\thanks{Research supported by the National Science Centre grant
  DEC-2012/05/N/ST1/02829. During part of the work on this paper the
  author was also supported by a scholarship of the Foundation for
  Polish Science. Part of the work on this paper was conducted during
  the author's internship at the Warsaw Center of Mathematics and
  Computer Science.}

\newtheorem{theorem}{Theorem}[section]
\newtheorem*{thm-irr}{Irreducibility}
\newtheorem*{thm-equiv}{Equivalence}
\newtheorem*{conjecture}{Conjecture}
\newtheorem{lemma}[theorem]{Lemma}
\newtheorem{proposition}[theorem]{Proposition}
\newtheorem{corollary}[theorem]{Corollary}
\theoremstyle{definition}

\theoremstyle{remark}
\newtheorem{remark}[theorem]{Remark}
\numberwithin{equation}{section}

\newcommand{\RR}{\mathbb{R}}
\newcommand{\ZZ}{\mathbb{Z}}

\newcommand{\CC}{\mathbb{C}}
\newcommand{\one}{\boldsymbol{1}}

\DeclareMathOperator{\diam}{diam}
\DeclareMathOperator{\pr}{pr}
\DeclareMathOperator{\Isom}{Isom}

\DeclareMathOperator{\Lip}{Lip}
\DeclareMathOperator*{\radlim}{rad\,lim}

\newcommand{\norm}[1]{\left\lVert#1\right\rVert}
\newcommand{\abs}[1]{\left\lvert#1\right\rvert}

\newcommand{\conj}{\overline}
\newcommand{\lasymp}{\prec}
\newcommand{\rasymp}{\succ}
\newcommand{\aless}{\apprle}
\newcommand{\agtr}{\apprge}

\newcommand{\bd}{\partial}
\newcommand{\tP}{\widetilde{P}}
\newcommand{\bG}{{\bd\Gamma}}
\newcommand{\wt}{\widetilde}

\begin{document}

\maketitle

\begin{abstract}
  Let $\Gamma$ be a Gromov hyperbolic group, endowed with an arbitrary
  left-invariant hyperbolic metric, quasi-isometric to a word
  metric. The action of $\Gamma$ on its boundary $\bG$ endowed with
  the Patterson-Sullivan measure $\mu$, after an appropriate
  normalization, gives rise to a faithful unitary representation of
  $\Gamma$ on $L^2(\bG,\mu)$. We show that these representations are
  irreducible, and
  give criteria for their unitary equivalence in terms of the metrics
  on $\Gamma$. Special cases include quasi-regular representations on
  the Poisson boundary.
\end{abstract}

\section{Introduction}
\label{sec:introduction} 

Any action of a group $G$ on a measure space $(X,\mu)$, preserving
the class of $\mu$, induces an action on the space of measurable
functions on $X$. It can be normalized to obtain a unitary
representation of $G$ on $L^2(X,\mu)$. This
construction generalizes the notion of a quasi-regular representation,
which we obtain when $X$ is a homogeneous space for $G$; we will still refer
to these representations as quasi-regular.

Irreducibility of such quasi-regular representations is a mixing-type
condition, strictly stronger than ergodicity. Indeed, for non-ergodic
actions the space $L^2(X,\mu)$ decomposes into spaces of functions
supported on the nontrivial invariant sets, and on the other hand, any
ergodic action of an abelian group, such as the action of $\ZZ$ on the
circle by powers of an irrational rotation, gives a reducible
representation. There are many natural examples of irreducible
quasi-regular representations:
\begin{itemize}
\item the natural action of the group of diffeomorphisms of a manifold
  $M$, or some of its subgroups preserving additional structure on $M$
  \cite{Vershik1975,Garncarek2010},
\item the action of the Thompson's groups $F$ and $T$ on the unit interval
  and the unit circle \cite{Garncarek2012,Dudko2015},
\item the action of a lattice of a Lie group on its Furstenberg
  boundary \cite{Bekka2002,Cowling1991}, 
\item the action of the automorphism group of a regular tree on its
  boundary \cite{Figa-Talamanca1991},
\item the action of a free group on its boundary
  \cite{Figa-Talamanca1983,Figa-Talamanca1994},
\item the action of the fundamental group of a compact strictly negatively
  curved Riemannian manifold $M$ on the boundary of the universal
  cover of $M$ \cite{Bader2011}.
\end{itemize}

The exact relationship between irreducibility of the quasi-regular
representation and the dynamical properties of the action of $G$ on
$X$ is fully understood only in the case of discrete groups acting on
discrete spaces \cite{Binder1993,Burger1997,Corwin1975}. The genuine
quasi-regular representations are also better understood, via the
notion of imprimitivity system \cite{Mackey1976}. For
general locally compact groups, irreducibility of the quasi-regular
representations was conjectured in \cite{Bader2011} for another broad
class of actions.

\begin{conjecture}
  For a locally compact group $G$ and a spread-out probability measure
  $\mu$ on $G$, the quasi-regular representation
  associated to the action of $G$ on the $\mu$-boundary of $G$ is irreducible.
\end{conjecture}

In this work we study the representations of hyperbolic groups
associated with actions on their Gromov boundaries endowed with the
Patterson-Sullivan measures. Following \cite{Bader2011}, we call them
\emph{boundary representations}. Our main result states that they are
always irreducible. Moreover, when the metric on the group is
quasi-isometrically perturbed, the class of the Patterson-Sullivan
measure varies, thus leading to a potentially vast supply of
non-equivalent irreducible representations. Indeed, we show that the
only unitary equivalences between the boundary representations arise
from rough similarities of the corresponding metrics. Our results thus
generalize the work of Bader and Muchnik \cite{Bader2011}.

The irreducibility of the quasi-regular representations can also be
seen in a slightly different light. As far as we know, this is the
first general construction of a family of faithful irreducible unitary
representations of an arbitrary hyperbolic group. In general,
providing such constructions for large classes of groups, for which
there is no structural description allowing to reduce the problem to
some better understood cases, seems to be a difficult task.

The line of our proof can be said to lie within bounded distance from
the arguments of Bader and Muchnik, which we generalize to the setting
of arbitrary hyperbolic groups, circumventing some of the difficulties
they had to deal with. Basically, we construct a family of operators
in the von Neumann algebra of the representation, analogous to the
operators used in their approach. However, since they try to obtain
them as weak operator limits of some arithmetic averages, in order
to prove convergence they need to resort to a result of Margulis,
describing the asymptotic behavior of the number of certain geodesic
segments on a manifold. By using weighted averages and choosing
suitable weights, we omit the necessity of knowing such asymptotics,
and obtain a more self-contained and simpler proof, applicable in a
wider context.

Recently, Uri Bader has informed us about an unpublished work of Roman
Muchnik, establishing irreducibility of quasi-regular representations
of hyperbolic groups associated with their actions on Poisson
boundaries of finitely supported symmetric random walks. This is also
a special case of our result, which we explain in
Section~\ref{sec:green-metr-poiss}. 

\subsection{Organization of the paper}
\label{sec:organization-paper}

In Sections~\ref{sec:preliminaries} and \ref{sec:geometric-setting} we
introduce some notational conventions and definitions from geometric
group theory, discuss some basic results concerning hyperbolic groups
and their boundaries, and finally define the class of representations
we are going to consider. All the geometry is contained in
Section~\ref{sec:shadows-cones}, where we explore some subsets of the
group, estimate their growth and show that they are nicely
distributed.  Section~\ref{sec:oper-posit-cone} uses these estimates
to construct certain operators in the von Neumann algebras of the
boundary representations. In Section~\ref{sec:properties} we gather
all the previous results into the proof of irreducibility of the
boundary representations. We also explain why they are weakly
contained in the regular
representation. Section~\ref{sec:classification} contains the
classification of the boundary representations with respect to unitary
equivalence. Finally, in Section~\ref{sec:applications} we discuss two
examples with more explicitly defined groups and metrics. We have a
closer look at the case of fundamental groups of negatively curved
manifolds, and we also explain how the conjecture mentioned in the
Introduction follows for a certain class of random walks on a
hyperbolic group.

The main text is followed by an appendix, in which we prove that the Patterson-Sullivan measures arising in our context are doubly ergodic. This result is needed in the proof of the Classification Theorem \ref{thm:inequivalence}. The proof of double ergodicity is based on ideas explained to us by Uri Bader and Alex Furman, and is a special case of a more general theorem, stating that the Patterson-Sullivan measures are isometrically doubly ergodic, whose proof will appear in their forthcoming paper \cite{Bader2014}.

\subsection{Acknowledgments}
\label{sec:acknowledgments}

We wish to thank Uri Bader, Michael Cowling, Alex Furman, Tadeusz
Ja\-nu\-sz\-kiewicz, and Tim Steger for their remarks and helpful
discussions. We are also grateful to Pawe{\l} J{\'o}ziak and Adam
Skalski for careful reading of the manuscript and their comments,
which helped to improve the text. Last but not least, we are indebted
to our advisor Jan Dymara for introducing us to the subject of
boundary representations, his ongoing support, and numerous
discussions.

\section{Preliminaries}
\label{sec:preliminaries}
In this section we introduce the basic notions associated with
hyperbolic spaces and groups. We start by fixing some notational
conventions for various kinds of estimates, which we will use
throughout the paper in order to avoid the aggregation of
non-essential constants and hopefully making the presentation more
lucid.  Then we introduce the basic terminology related to
quasi-isometries, define hyperbolic spaces and groups, and finally,
discuss the notion of the Gromov boundary. For details on these
subjects see \cite[Chapters III.H.1 and III.H.3]{Bridson1999}.

\subsection{Estimates}
\label{sec:estimates}

In the paper we will work with additive and multiplicative
estimates. In order to avoid the escalation of constants coming from
such estimates, we will suppress them using the following
notation. Let $f,g$ be functions on a set $X$. If there exists $C>0$
such that $f(x) \leq Cg(x)$ for all $x$, we write $f \lasymp g$. If
both $f \lasymp g$ and $g \lasymp f$ hold, we write $f \asymp
g$. Analogously, for additive estimates, $f \aless g$ if there exists
$c$ such that $f \leq g + c$, and $f \approx g$ if both $ f \aless g$
and $g\aless f$ hold. The variables in which the estimates are assumed
to be uniform will be either clear from context or explicitly
mentioned. Sometimes, to indicate that we do not care whether the
estimate is uniform in some of the variables (which does not mean that
we claim it is not), we write them as subscripts to the symbol of the
corresponding estimate, e.g.\ $f(x,y) \aless_x g(x,y)$ need not be
uniform in $x$.

\subsection{Quasi-isometries}
\label{sec:quasi-isometries}
Let $(X,d_X)$ and $(Y,d_Y)$ be metric spaces. Take $L\geq 1$ and
$C\geq0$. A map $\phi\colon X\to Y$ satisfying the condition
\begin{equation}
  \frac{1}{L}d_X(p,q) - C \leq d_Y(\phi(p),\phi(q)) \leq Ld_X(p,q) + C
\end{equation}
for all $p,q\in X$ is called an \emph{(L,C)-quasi-isometric embedding}. If the image of $\phi$ is a $C$-net in $Y$, i.e.\ its
$C$-neighborhood covers $Y$, or equivalently, if there exists a
quasi-isometric embedding $\psi\colon Y\to X$, called the
\emph{quasi-inverse of $\phi$}, such that $d_X(x, \psi\phi(x))$ and
$d_Y(y,\phi\psi(y))$ are uniformly bounded functions on $X$ and $Y$
respectively, then $\phi$ is an \emph{(L,C)-quasi-isometry}. A
$(1,C)$-quasi-isometry is called a \emph{$C$-rough isometry}. A
quasi-isometry $\phi$ satisfying $d_Y(\phi(x),\phi(y))\approx L
d_X(x,y)$ with additive constant $C$ is an \emph{(L,C)-rough
  similarity}.

An $(L,C)$-quasi-isometric embedding $\gamma\colon \RR\to X$ is called
an \emph{(L,C)-quasi-geodesic} in $X$. Similarly one defines
quasi-geodesic rays and segments, and their roughly geodesic variants.
We say that $X$ is an \emph{$(L,C)$-quasi-geodesic space}, if any two
points in $X$ can be joined by an $(L,C)$-quasi-geodesic segment. A
\emph{$C$-roughly geodesic space} is defined in the same manner. We
will later fix the constants $L$ and $C$ and suppress them from
notation.

\subsection{Hyperbolic spaces and groups}
\label{sec:hyperbolic-spaces-and-groups}
Let $(X,d)$ be a metric space. For any base-point $o\in X$ one defines
the Gromov product $(\cdot,\cdot)_o\colon X\times X \to [0,\infty)$
with respect to $o$ as
\begin{equation}
  (x,y)_o = \frac{1}{2}(d(x,o)+d(y,o)-d(x,y)).
\end{equation}
A different choice of the base-point leads to another Gromov product, satisfying
\begin{equation}
  \label{eq:Gromov-product-basepoint-change}
  \abs{(x,y)_o - (x,y)_p} \leq d(o,p).
\end{equation}
If the Gromov product on $X$ satisfies the estimate
\begin{equation}
  \label{eq:hyperbolicity-inequality}
  (x,y)_o \agtr \min\{ (x,z)_o, (y,z)_o \}
\end{equation}
for some (equivalently, for every---but with a different constant)
base-point $o\in X$, the space $X$ is said to be \emph{hyperbolic}. We
may iterate~\eqref{eq:hyperbolicity-inequality}, to obtain
\begin{equation}
  (x_1,x_n)_o \agtr \min\{(x_1,x_2)_o,(x_2,x_3)_o,\dots,(x_{n-1},x_n)_o\},
\end{equation}
with constants depending only on $n$. The property of being hyperbolic
is preserved by quasi-isometries within the class of geodesic
spaces. In case of general metric spaces, it is possible to
quasi-isometrically perturb a hyperbolic metric and obtain a
non-hyperbolic one (see~\cite[Proposition A.11]{Blachere2011}).

A finitely generated group $\Gamma$ is hyperbolic if its Cayley graph
with respect to some finite set of generators is hyperbolic. As Cayley
graphs of a given group are geodesic and quasi-isometric to each
other, this notion does not depend on the generating set.  The
quasi-isometric metrics induced on the group by the path metrics on
its Cayley graphs are called the \emph{word metrics}.  We will denote
by $\mathcal{D}(\Gamma)$ the class of all hyperbolic left-invariant
metrics on $\Gamma$ (not necessarily coming from an action on a
geodesic space), quasi-isometric to a word metric through the
identity map of \(\Gamma\). Finally, a hyperbolic group is
\emph{non-elementary} if it does not contain a cyclic subgroup of
finite index.

\subsection{The Gromov boundary}
\label{sec:gromov-boundary}

Now, assume that $X$ is hyperbolic and has a fixed base-point $o\in X$,
which we will omit in the notation for the Gromov product. We will
also denote $\abs{x} = d(x,o)$. A sequence $(x_n)\subset X$
\emph{tends to $\infty$} if
\begin{equation}
\lim_{i,j\to\infty} (x_i,x_j) =\infty.
\end{equation}
Two such sequences $(x_n)$ and $(y_n)$ are \emph{equivalent} if
$\lim_{n\to\infty}(x_n,y_n)=\infty$. By~\eqref{eq:Gromov-product-basepoint-change}
these notions are independent of the base-point. The boundary of
$X$, denoted $\bd X$, is the set of equivalence classes of sequences
tending to infinity.  The space $\overline{X}=X\cup\bd X$ can be given
a natural topology making it a compactification of $X$, on which the
isometry group $\Isom(X)$ acts by homeomorphisms.

The Gromov product can be extended (in a not necessarily continuous way) to $\overline{X}$ in such a way
that the estimate~\eqref{eq:hyperbolicity-inequality} is still
satisfied (with different constants). One simply represents elements
of $X$ as constant sequences, and for $x,y\in\overline{X}$ defines
\begin{equation}
  (x,y) = \sup\liminf_{i,j\to\infty} (x_i,y_j),
\end{equation}
where the supremum is taken over all representatives $(x_i)$ and
$(y_i)$ of $x$ and $y$. A sequence $(x_i)\subset X$
converges to $\xi\in\bd X$ if and only if $(x_i,\xi)\to\infty$, so in
particular, representatives of $\xi$ are exactly the sequences in $X$
converging to $\xi$. By~\cite[Remark 3.17]{Bridson1999}, we have
\begin{equation}
  \label{eq:extended-gromov-product-representative-estimate}
  \liminf_{i,j\to\infty}(x_i,y_j)\approx(x,y)
\end{equation}
whenever $x_i\to x$ and $y_i\to y$.

The topology of $\bd X$ is metrizable. For sufficiently small
$\epsilon > 0$ there exists a metric $d_\epsilon$ on $\bd X$,
compatible with its topology, satisfying
\begin{equation}
  d_\epsilon(\xi,\eta) \asymp_\epsilon e^{-\epsilon(\xi,\eta)}.
\end{equation}
Such a metric is called a \emph{visual metric}.

We will later use the fact that for a hyperbolic group $\Gamma$ the
only $\Gamma$-equivariant homeomorphism $\phi$ of $\bG$ is the
identity map. It follows from the fact that any element of $\Gamma$ of
infinite order has exactly one attracting point in $\bG$, which is
therefore fixed by $\phi$, and the attracting points of all such
elements form a dense subset \cite[Proposition 4.2 and Theorem
4.3]{IlyaKapovich}.

\section{The geometric setting}
\label{sec:geometric-setting}

In this section we describe some ways to deal with non-geodesic
hyperbolic metrics. As the representations we will consider depend on
the metric on the group, this will allow to investigate a class of
representations much wider than those obtained from the word
metrics. Everything we need in this regard is contained in the
papers~\cite{Blachere2011,Bonk2000}.

In~\cite{Blachere2011} the notions of a quasi-ruler and quasi-ruled
space are introduced, and the fundamental properties of the
Patterson-Sullivan measures for quasi-ruled hyperbolic spaces,
generalizing the results of \cite{Coornaert1993}, which apply only to
metrics coming from proper actions on geodesic spaces, are
developed. The article~\cite{Bonk2000} studies boundaries of almost
geodesic hyperbolic spaces, and is a useful reference for some basic lemmas.

It turns out that the classes of hyperbolic quasi-ruled spaces and
hyperbolic almost geodesic spaces are the same and equal to the class
of roughly geodesic hyperbolic spaces. We discuss the notion of a
quasi-ruled space only in order to formulate
Theorem~\ref{thm:bhm}. Afterwards, all the arguments will be based on
the notion of a rough geodesic.

\subsection{Roughly geodesic hyperbolic spaces}
\label{sec:roughly-geod-hyperb}

For $\tau\geq0$ a \emph{$\tau$-quasi-ruler} is a quasi-geodesic
$\gamma\colon\RR\to X$ satisfying for all $s<t<u$ the condition
\begin{equation}\label{eq:def-quasi-ruler}
  (\gamma(s),\gamma(u))_{\gamma(t)} \leq \tau.
\end{equation}
The space $X$ is said to be \emph{$(L,C,\tau)$-quasi-ruled} if it is a
$(L,C)$-quasi-geodesic space, and every $(L,C)$-quasi-geodesic is a
$\tau$-quasi-ruler. By \cite[Theorem
A.1]{Blachere2011}, if $\phi\colon X\to Y$ is a quasi-isometry with
$X$ hyperbolic and geodesic, then $Y$ is hyperbolic if and only if it
is quasi-ruled. It follows that for a hyperbolic group $\Gamma$, all
the metrics in the class $\mathcal{D}(\Gamma)$ are quasi-ruled.

By~\cite[Lemma A.2]{Blachere2011}, for every $L$, $C$, and $\tau$
there exists $K>0$ such that every $(L,C,\tau)$-quasi-ruled space is
$K$-roughly geodesic. On the other hand, it is clear that a
$K$-roughly geodesic space is $(1,K,3K/2)$-quasi-ruled. By \cite[Proposition
5.2(1)]{Bonk2000}, the hyperbolic spaces studied therein are also
exactly the roughly geodesic hyperbolic spaces.

Now suppose that $X$ is a roughly geodesic hyperbolic space. Every
roughly geodesic ray $\gamma$ in $X$ converges to an endpoint
$\gamma(\infty)$ in the boundary. The converse statement is also true,
i.e.\ every point in $\bd X$ is the endpoint of some $K$-roughly
geodesic ray, where $K$ depends only on $X$ \cite[Proposition
5.2(2)]{Bonk2000}. In a similar fashion, every pair of distinct points
of $\bd X$ can be joined by a $K$-roughly geodesic line
\cite[Proposition 5.2(3)]{Bonk2000}. From now on, when we use the
terms \emph{roughly geodesic segment/ray/line} without specifying the
constant, we always think of the universal constants from the
definition of a roughly geodesic space and the remark above.

By \cite[Proposition 5.5]{Bonk2000}, if the map $\phi\colon X\to Y$ is a
$(L,C)$-quasi-isometry of roughly geodesic hyperbolic spaces, their
Gromov products satisfy the estimates
\begin{equation}
  \frac{1}{L} (x,y)_z \aless (\phi(x),\phi(y))_{\phi(z)} \aless
  L(x,y)_z
\end{equation}
 uniformly for all $x,y,z\in X$. As a consequence, for a
hyperbolic group $\Gamma$ all the metrics in $\mathcal{D}(\Gamma)$
give rise to exactly the same boundary.

\subsection{Quasi-conformal measures}
\label{sec:quasi-conf-meas}

Consider a roughly geodesic hyperbolic space $X$ with a base-point
$o\in X$, and a non-elementary hyperbolic group $\Gamma\subseteq
\Isom(X)$ which acts on $X$ properly and cocompactly.  A measure $\mu$
on $(\bd X,d_\epsilon)$ is said to be \emph{$\Gamma$-quasi-conformal
  of dimension $D$} if it is quasi-invariant under the action of
$\Gamma$, and the corresponding Radon-Nikodym derivatives satisfy the
estimate
\begin{equation}
  \frac{dg_*\mu}{d\mu}(\xi) \asymp e^{\epsilon D (2 (go,\xi)-d(o,go))}
\end{equation}
uniformly in $\xi$ and
$g$. Since $d(o,go)\approx_{o,p}d(p,gp)$, this notion is independent of the
choice of $o$. Moreover, being $\Gamma$-quasi-conformal does not depend
on $\epsilon$, as for different values of $\epsilon$ only the
dimension $D$ changes. Finally, $\mu$ is \emph{Ahlfors regular of dimension $D$} if it
satisfies the estimate
\begin{equation}
  \mu(B_{\bd X}(\xi,\rho))\asymp \rho^D
\end{equation}
uniformly in $\xi$ and $\rho\leq\diam{\bd X}$. In particular, since $\bd X$ is compact,
any Ahlfors regular measure on $\bd X$ is finite.

Recall that the Hausdorff measure of a metric space $Y$ is
defined as follows. First, for $\alpha\geq 0$ one defines the
$\alpha$-dimensional Hausdorff measure $\mathcal{H}_\alpha$ as
\begin{equation}
  \mathcal{H}_\alpha(E) = \lim_{\theta\to 0^+}\inf\Big\{ \sum_i (\diam{U_i})^\alpha
  : E\subseteq\bigcup_i U_i\;\text{and}\;\diam{U_i}\leq \theta\Big\}
\end{equation}
for every Borel set $E\subseteq Y$. Then, the Hausdorff dimension of $Y$ is the number
\begin{equation}
  \dim_H Y = \inf\{\alpha : \mathcal{H}_\alpha(Y)=0\} = \sup\{\alpha :
  \mathcal{H}_\alpha(Y) = \infty\}. 
\end{equation}
The Hausdorff measure on $Y$ is the $(\dim_H Y)$-dimensional Hausdorff
measure. See \cite{Rogers1970} for details on Hausdorff measures.

Now, take $x\in X$ and denote
\begin{equation}
  \label{eq:def-D}
  D = \limsup_{R\to\infty} \frac{1}{\epsilon R}\log\abs{B_X(x,R)\cap \Gamma x},
\end{equation}
and $\omega = e^{D\epsilon}$. We then have the following.
\begin{theorem}[{\cite[Theorem 2.3]{Blachere2011}}]
  \label{thm:bhm}
  Suppose that $X$ is a proper roughly geodesic hyperbolic space, and
  $\Gamma\subseteq\Isom(X)$ is a non-elementary hyperbolic group,
  acting properly and cocompactly. Then the Hausdorff dimension of
  $(\bd X, d_\epsilon)$ is equal to $D$, defined in \eqref{eq:def-D},
  and the corresponding Hausdorff measure $\mu$ is
  $\Gamma$-quasi-conformal of dimension $D$ and Ahlfors regular of
  dimension $D$. Furthermore, any $\Gamma$-quasi-conformal measure
  $\mu'$ on $\bd X$ is equivalent to $\mu$ with Radon-Nikodym
  derivative $d\mu'/ d\mu \asymp 1$ a.e., and $ \abs{B_X(x,R)\cap
    \Gamma x} \asymp \omega^R$.
\end{theorem}

In particular, this theorem implies that the quasi-conformal measures
associated to different choices of $\epsilon$ are equivalent, so the
above considerations lead to a unique measure class on $\bd X$ (in
fact a class of the finer relation of equivalence with Radon-Nikodym
derivatives bounded away from $0$ and $\infty$), depending only on the
metric $d$, called the \emph{Patterson-Sullivan class}. Also, by
Ahlfors regularity, the boundary has no isolated points.

We say that a measure class preserving action of a group $G$ on a
measure space $(X,\nu)$ is \emph{doubly ergodic}, if the induced
diagonal action of $G$ on $(X^2,\nu^2)$ is ergodic. In the
classification of the boundary representations, double ergodicity of
Patterson-Sullivan measures will be crucial. This result is known to
experts, but apparently the proof has never been written down. It was
communicated to us by Uri Bader that the full proof of a stronger
property called \emph{double isometric ergodicity} will appear in a
forthcoming joint paper with Alex Furman
\cite{Bader2014}. We include the proof of double ergodicity, based on the ideas explained to us by Bader and Furman, in Appendix~\ref{sec:double-ergod-patt-1}.

\subsection{Boundary representations}
\label{sec:bound-repr}

We will now fix some notation for the rest of the paper. Let $\Gamma$ be a
non-elementary hyperbolic group. Fix a metric
$d\in\mathcal{D}(\Gamma)$, and choose $1\in\Gamma$ as the
base-point. Since $\Gamma$ acts on itself by isometries freely and
cocompactly, we are in the setting of
Section~\ref{sec:quasi-conf-meas}. Pick a sufficiently small
$\epsilon>0$, and let $D$ be the Hausdorff dimension of
$(\bG,d_\epsilon)$. Denote by $\mu$ the corresponding Hausdorff
measure. We may normalize $\mu$ and $d_\epsilon$ in such a way that
$\mu(\bG)=1$ and $\diam\bG=1$. Since $D\epsilon$ is constant, by choosing sufficiently small $\epsilon$, we may also
assume that $D>1$. Now, denote
\begin{equation}
  \label{eq:Pg-def-estimate}
  P_g(\xi)=\frac{dg_*\mu}{d\mu}(\xi) \asymp \omega^{2(g,\xi)-\abs{g}}.
\end{equation}
The \emph{boundary representation} $\pi$ of $\Gamma$ associated to $\mu$
is the unitary representation of $\Gamma$ on the Hilbert space
$L^2(\bG,\mu)$ given by
\begin{equation}
\label{eq:def-boundary-rep}
  [\pi(g)\phi](\xi) = P_g^{1/2}(\xi)\phi(g^{-1}\xi)
\end{equation}
for $\phi\in L^2(\bG,\mu)$ and $g\in\Gamma$. If we take a measure
$\nu$ equivalent to $\mu$, then the unitary isomorphism $T_{\mu\nu}\colon L^2(\bG,\mu)\to
L^2(\bG,\nu)$ defined by
\begin{equation}
  T_{\mu\nu}\phi = \left(\frac{d\mu}{d\nu}\right)^{1/2}\phi
\end{equation}
intertwines the corresponding boundary representations. We therefore
obtain a unique (up to unitary equivalence) representation of $\Gamma$
associated to the class of $\Gamma$-quasi-conformal measures on $\bG$
with respect to $d$.

\section{Shadows and cones}
\label{sec:shadows-cones}

In this section we will work with $\Gamma$ in order to estimate the
cardinalities of some of its subsets. First, we introduce the
classical notion of the shadow cast by an element of the group onto
its boundary. Then, for a ball $B$ in the boundary, we define the cone
over $B$ as the set of all elements $\Gamma$ whose shadows intersect
$B$. It turns out that the growth of such a cone behaves as one could
expect, i.e.\ the cardinality of its intersection with a large ball in
$\Gamma$ is approximately $\mu(B)$ times the cardinality of the ball.

We then move on to define double shadows in $\bG^2$. A double
shadow of $g$ is the product of suitable shadows of $g$ and $g^{-1}$. We
show that they form a nice cover of $\bG^2$, just as
in the case of ordinary shadows in $\bG$.

\subsection{Shadows}
\label{sec:shadows}

We begin by observing that for any element $g$ of $\Gamma$ there
exists a roughly geodesic ray emanating from $1$ and passing within a uniform
distance from $g$.  In terms of the Gromov product this can be stated
as follows.
\begin{lemma}
  \label{lem:near-geodesic-ray}
    The estimate 
    \begin{equation}
      \sup_{\xi\in\bd\Gamma}(g,\xi) \approx \abs{g}
    \end{equation}
    holds uniformly for $g\in\Gamma$.
\end{lemma}

\begin{proof}
  We get the upper estimate $\sup(g,\xi) \leq \lvert{g}\rvert$ from
  the triangle inequality. 

  The Gromov product on $\Gamma$ satisfies the identity $(g,h) +
  (g^{-1},g^{-1}h) = \lvert{g}\rvert$, which, after extension to
  $\overline{\Gamma}$, takes the form
  \begin{equation}
    \label{eq:gromov-product-inverse}
     (g^{-1}, g^{-1}\xi) \approx \abs{g} - (g,\xi).
  \end{equation}
  If we fix two distinct points $\xi_1,\xi_2\in\bG$, then 
  \begin{equation}
    \max_i (g^{-1},g^{-1}\xi_i) \approx \abs{g} - \min_i (g,\xi_i)
    \agtr \abs{g}-(\xi_1,\xi_2),
  \end{equation}
  which gives the estimate from below for $g^{-1}$.
\end{proof}

Using Lemma~\ref{lem:near-geodesic-ray}, for every $g\in\Gamma$ we may
fix $\hat{g}\in\bd\Gamma$ such that $(g,\hat{g}) \approx
\lvert{g}\rvert$.  We will also denote by $\check{g}$ the point in the
boundary corresponding to $g^{-1}$. The point $\hat{g}$ plays the same
role as the endpoint of the geodesic ray starting at the base-point and
passing through $g$ in a CAT(0) space. In particular, we have
\begin{equation}
  \label{eq:hat-g-xi}
  (\xi,g) \approx \min\{\abs{g},(\hat{g},\xi)\}
\end{equation}
for all $\xi\in\Gamma$. This estimate will be usually used in the form
\begin{equation}
  \label{eq:omega-hat-g-xi}
  \omega^{(\xi,g)} \asymp \min\{\omega^{\abs{g}}, d_\epsilon(\hat{g},\xi)^{-D}\}.
\end{equation}

By Theorem~\ref{thm:bhm}, the growth of
$\Gamma$ satisfies  $\abs{B_\Gamma(1,R)}\asymp
\omega^R$. Fix $r>0$ and for $R > 0$ define the annulus
\begin{equation}
  A_{R} = \{ g\in\Gamma : R-r \leq \abs{g} \leq R+r\},
\end{equation}
We will assume that $r$ is sufficiently large for the following three
conditions to hold:
\begin{enumerate}[(1)]
\item the lower bound for the cardinality of the ball $B_\Gamma(1,R+r)$
  exceeds the upper bound for the cardinality of $B_\Gamma(1,R-r)$, so
  that the annuli $A_{R}$ satisfy the growth estimate
  $\abs{A_{R}}\asymp \omega^R$,
\item for any roughly geodesic\footnote{recall, that by this we mean a \(K\)-roughly geodesic ray with \(K\) fixed in Section~\ref{sec:roughly-geod-hyperb}} ray $\gamma$ from $1$ to $\xi\in\bG$ we
  have $\gamma(R)\in A_R$,
\item $r$ satisfies the bound obtained\footnote{To precisely
    formulate condition (3) we need
    Lemma~\ref{lem:preventing-cancellation}, which asserts that there
    exists some universal constant $\tau$ related to cancellations of
    elements in the group. Until
    Proposition~\ref{prop:double-shadows-cover}, only conditions (1)
    and (2) will be used, so the reader need not be afraid of a
    circular definition. We could just postpone condition (3) and require it in the proof of Proposition~\ref{prop:double-shadows-cover}, as we can always take larger \(r\).} in the proof of
  Proposition~\ref{prop:double-shadows-cover}, ensuring that the elements of $\Gamma$ of
  length approximately $R$ constructed therein are in $A_R$.
\end{enumerate}

For $\sigma>0$ define the
\emph{shadow} $\Sigma(g,\sigma)$ of $g$ as the closed ball
\begin{equation}
  \Sigma(g,\sigma) = B_{\bd\Gamma}(\hat{g}, e^{-\epsilon(\abs{g}-\sigma)}).
\end{equation} 
Usually one defines the shadow of \(g\) in a slightly different way, by taking the set of endpoints of geodesic rays passing through a ball of fixed radius centered at \(g\). These shadows can be uniformly sandwiched between our shadows, and we find defining shadows to be balls in the boundary to be more suitable for our applications.

The following fundamental property of shadows is classical, and since we use a non-classical definition of a shadow, we
include its very short proof. This standard lemma has also a second
part, saying that the multiplicity of the cover of $\bG$ by shadows
is uniformly bounded in $R$, but we will not need that statement.
\begin{lemma}
  \label{lem:number-of-shadows}
  For sufficiently large $\sigma$, the family of shadows
  $\{\Sigma(g,\sigma):g\in A_{R}\}$ is a cover of $\bd\Gamma$ for any $R\geq0$.
\end{lemma}

\begin{proof}
  For $\xi\in\bd\Gamma$ take a roughly geodesic ray $\gamma$ from $1$
  to $\xi$. Then
  $g=\gamma(R)\in A_{R}$ and we have $(\hat{g},\xi)\agtr\min\{(\hat{g},g),(g,\xi)\}\approx
  R$, so for sufficiently large $\sigma$ we get $\xi\in\Sigma(g,\sigma)$.\qedhere
\end{proof}

We may now define $\Sigma(g)=\Sigma(g,\sigma)$ with $\sigma$ sufficiently
large to satisfy the conclusion of
Lemma~\ref{lem:number-of-shadows}. 

\subsection{Cones over balls in the boundary}
\label{sec:cones-over-balls}

For $\xi\in\bG$ and $\theta>0$ define
the \emph{cone over $B_{\bG}(\xi,e^{-\epsilon\theta})= B_{\bG}(\xi,\omega^{-\theta/D})$} as
\begin{equation}
  C(\xi,\theta) = \{g\in\Gamma : \Sigma(g)\cap
  B_{\bG}(\xi,e^{-\epsilon\theta}) \ne\emptyset\},
\end{equation}
and denote
\begin{equation}
  C_{R}(\xi,\theta) = A_{R}\cap C(\xi,\theta).
\end{equation}

\begin{lemma}
  \label{lem:growth-of-cones}
  The growth of the cone $C(\xi,\theta)$ satisfies the estimates
  \begin{equation}
    \omega^{R-\theta} \lasymp \abs{C_{R}(\xi,\theta)} \lasymp \omega^R
  \end{equation}
  uniformly in $R,\theta,\xi$. When $R\geq\theta$, the tighter estimate
  \begin{equation}
    \abs{C_{R}(\xi,\theta)} \asymp \omega^{R-\theta}
  \end{equation}
  holds.
\end{lemma}

\begin{proof}
  The upper bound $\abs{C_{R}(\xi,\theta)}\lasymp\omega^R$ follows
  from the estimate on $\abs{A_{R}}$. For the lower bound, observe
  that by Lemma~\ref{lem:number-of-shadows} the shadows $\Sigma(g)$
  of $g\in C_{R}(\xi,\theta)$ cover the ball
  $B(\xi,e^{-\epsilon\theta})$. Hence,
  \begin{equation}
    \omega^{-\theta}\asymp\mu(B_{\bG}(\xi,e^{-\epsilon\theta})) \leq
    \sum_{g\in C_{R}(\xi,\theta)}\mu(\Sigma(g))\asymp\abs{C_{R}(\xi,\theta)}\omega^{-R},
  \end{equation}
  so $\abs{C_R(\xi,\theta)} \rasymp \omega^{R-\theta}$. 
  
  Now, assume that $R\geq\theta$. Let $\gamma$ be a roughly geodesic ray
  from $1$ to $\xi$.  If $g\in
  C_{R}(\xi,\theta)$, we may pick some $\eta\in\Sigma(g)\cap
  B_{\bG}(\xi,e^{-\epsilon\theta})$. We then have
  \begin{equation}
    (g,\gamma(\theta)) \agtr \min\{ (g,\hat{g}),
    (\hat{g},\eta),(\eta,\xi),(\xi,\gamma(\theta))\} \agtr  \theta,
  \end{equation}
  and in consequence
  \begin{equation}
    d(g,\gamma(\theta)) \aless R-\theta.
  \end{equation}
  Therefore, $C_{R}(\xi,\theta)\subseteq B_\Gamma(\gamma(\theta),R-\theta+C)$ for some constant
  $C$, and the last estimate follows from the bound on the growth of $\Gamma$.
\end{proof}

\subsection{Shadows in the square of the boundary}
\label{sec:shad-square}

The next lemma will allow us to understand the distribution of the
points $(\hat{g},\check{g})$ in $\bG^2$. It generalizes the
observation that if we take two elements $g,h$ of a non-abelian free
group expressed in the standard generators, then after possibly changing
the last letter of $g$, there is no cancellation in the product
$gh$. We thought that such a natural result should be well-known, but
to our surprise we did not find it in any of the standard references
for hyperbolic groups. Therefore we present it together with its full
proof.

\begin{lemma}
  \label{lem:preventing-cancellation}
  Let $\Gamma$ be a non-elementary hyperbolic group, endowed with a
  metric $d\in\mathcal{D}(\Gamma)$. There exists $\tau > 0$ such that
  for any $g_0,h\in\Gamma$ one can find $g\in B_\Gamma(g_0,\tau)$ such
  that $\abs{gh}\geq \abs{g}+\abs{h}-2\tau$.
\end{lemma}

\begin{proof} 
  For every $g\in\Gamma$ fix a roughly geodesic segment
  $\gamma_g\colon[0,\abs{g}]\to \Gamma$ joining $1$ to $g$, and its
  reverse $\overline{\gamma}_g(t)=\gamma_g(\abs{g}-t)$.  Now, take any
  $\tau>0$ and suppose that for all $g\in B_\Gamma(g_0,\tau)$ we have
  the opposite inequality $\abs{gh}<\abs{g}+\abs{h}-2\tau$, or equivalently
  $(1,gh)_g > \tau$. In particular, this implies that
  $\abs{g},\abs{h}>\tau$, and we may compute, with estimates being
  uniform in $\tau$:
  \begin{equation}
    (\overline{\gamma}_g(\tau),g\gamma_h(\tau))_g \agtr \min\{ (\overline{\gamma}_g(\tau), 1)_g,
    (1,gh)_g,(gh,g\gamma_h(\tau))_g\} \agtr \tau,
  \end{equation}
  and in consequence, $d(\overline{\gamma}_g(\tau),g\gamma_h(\tau))
  \approx 0$. On the other hand, $(g,g_0) \geq \abs{g} - \tau$, so
  \begin{multline}
    (\overline{\gamma}_g(\tau),\gamma_{g_0}(\abs{g}-\tau)) \agtr \\ \agtr
    \min\{ (\overline{\gamma}_g(\tau),g), (g,g_0),
    (g_0,\gamma_{g_0}(\abs{g}-\tau))\} \approx \abs{g}-\tau,
  \end{multline}
  and thus $d(\overline{\gamma}_g(\tau),\gamma_{g_0}(\abs{g}-\tau))
  \approx 0$. Finally, we obtain
  \begin{equation}
      d(g\gamma_h(\tau),{\gamma}_{g_0}(\abs{g}-\tau))  \leq
      d(g\gamma_h(\tau),\overline{\gamma}_g(\tau)) +
      d(\overline{\gamma}_g(\tau),\gamma_{g_0}(\abs{g}-\tau))  \approx 0.
  \end{equation}
  It follows that the injective map $g\mapsto g\gamma_h(\tau)$ sends
  the ball $B_\Gamma(g_0,\tau)$ into a fixed radius neighborhood of
  the interval $\overline{\gamma}_{g_0}([0,2\tau])$. Since $\Gamma$ is
  non-elementary, the volume of the ball grows exponentially with
  $\tau$, while the neighborhood of the interval
  $\overline{\gamma}_{g_0}([0,2\tau])$ has linear growth, hence for sufficiently
  large $\tau$, independent of $g_0$ and $h$, we obtain a
  contradiction.
\end{proof}

For $\sigma>0$ we now define the \emph{double shadow} of $g\in\Gamma$ as
\begin{equation}
  \Sigma_2(g,\sigma) = B_\bG(\hat{g},e^{-\epsilon(\abs{g}/2-\sigma)})\times
  B_\bG(\check{g},e^{-\epsilon(\abs{g}/2-\sigma)}) \subseteq \bG^2.
\end{equation}
Thanks to the factor of $1/2$ in the exponent, the measure of a double
shadow of $g\in A_R$ is approximately proportional to $1/\lvert
A_R\rvert$. Just as ordinary shadows, the double shadows of elements
of $A_R$ form a cover.

\begin{proposition}
  \label{prop:double-shadows-cover}
  For sufficiently large $\sigma>0$ the family $\{ \Sigma_2(g,\sigma) :
  g\in A_R\}$ of double shadows is a cover of\/ $\bG^2$ for all $R>0$.
\end{proposition}

\begin{proof}
  Take $(\xi_1,\xi_2)\in\bG^2$, and for $i=1,2$ let $\gamma_i$ be a roughly
  geodesic ray from $1$ to $\xi_i$. Put $g_i=\gamma_i(R/2)$. By
  Lemma~\ref{lem:preventing-cancellation} there exist a universal
  constant $\tau$ and $g\in B_\Gamma(g_1,\tau)$ such that $\lvert
  gg_2^{-1}\rvert\approx R$, and, as it was mentioned in the
  definition of the annulus $A_R$, we may assume that its thickness is
  sufficiently large for it to contain the element $gg_2^{-1}$. We
  have $(gg_2^{-1},g) \approx R/2$ and $(g_2g^{-1},g_2) \approx R/2$, and
  in consequence
  \begin{equation}
    (gg_2^{-1},\xi_1) \agtr \min\{ (gg_2^{-1},g) , (g,g_1),
    (g_1,\xi_1)\} \agtr R/2 
  \end{equation}
  and
  \begin{equation}
    (g_2g^{-1},\xi_2) \agtr \min \{ (g_2g^{-1},g_2), (g_2,\xi_2)\}
    \agtr R/2.
  \end{equation}
  Hence, for sufficiently large $\sigma$, the double shadow
  $\Sigma_2(gg_2^{-1},\sigma)$ contains the pair $(\xi_1,\xi_2)$.
\end{proof}
Similarly as in the case of shadows, we will denote
$\Sigma_2(g)=\Sigma_2(g,\sigma)$ for some fixed $\sigma$ sufficiently
large for Proposition~\ref{prop:double-shadows-cover} to hold.

\section{Operators in the positive cone}
\label{sec:oper-posit-cone}

By the \emph{positive cone of the representation $\pi$} we will
understand the weak operator closure in $\mathcal{B}(L^2(\bG,\mu))$ of
the set of linear combinations of elements of $\pi(\Gamma)$ with
positive coefficients. The purpose of this section is to prove
Proposition~\ref{prop:measure-nuclear-operator-in-cone}, which states
that operators arising from positive kernels in $L^\infty(\bG^2)$ are
contained in the positive cone of $\pi$. The operators in question
will be constructed as weak operator limits of sequences of weighted
averages of normalized operators $\pi(g)$ with $g\in A_R$. Convergence
will be first tested on Lipschitz functions, and then established
using density of Lipschitz functions in $L^2$, and uniform boundedness
of the averages.

\subsection{Uniform boundedness of averages of $P_g^{1/2}$}
\label{sec:unif-bound-aver}

Recall that $P_g=dg_*\mu/d\mu$. Let us define 
\begin{equation}
  \tP_g=\frac{P_g^{1/2}}{\lVert{P_g^{1/2}}\rVert_1}.
\end{equation}
We begin by finding an estimate for the norm $\lVert
P_g^{1/2}\rVert_1$, and using it to get a more manageable
approximation of the function $\tP_g$.
\begin{lemma}
  \label{lem:poisson-L1-norm-estimate}
  The $L^1$-norms of $P_g^{1/2}$ satisfy the estimate
  \begin{equation}
    \lVert P_g^{1/2} \rVert_1 \asymp \omega^{-\abs{g}/2}(1+\abs{g})
  \end{equation}
  uniformly in $g$. Moreover,
  \begin{equation}
     \tP_g(\xi)\asymp \frac{\omega^{(g,\xi)}}{1+\abs{g}} \lasymp \frac{d_\epsilon(\hat{g},\xi)^{-D}}{1+\abs{g}}.
  \end{equation}
\end{lemma}

\begin{proof}
  By estimate~\eqref{eq:hat-g-xi}, we have
  \begin{equation}
    P_g^{1/2}(\xi)\asymp \omega^{(g,\xi)-\abs{g}/2} \asymp
    \omega^{-\abs{g}/2} \min\{\omega^{\abs{g}}, d_\epsilon(\hat{g},\xi)^{-D}\}.
  \end{equation}
  Using Ahlfors regularity and the Fubini's theorem, we calculate
  \begin{equation}
    \begin{split}
      \omega^{\abs{g}/2} \lVert P_g^{1/2} \rVert_1 & \asymp \int_{\bG}
      \min\{\omega^{\abs{g}},
      d_\epsilon(\hat{g},\xi)^{-D}\}\,d\mu(\xi) = \\
      & = \int_0^{\omega^{\abs{g}}} \mu\{ \xi :
      d_\epsilon(\hat{g},\xi)^{-D} > t \}\,dt = \\
       & = 1 + \int_{1}^{\omega^{\abs{g}}} \mu\{ \xi :
      d_\epsilon(\hat{g},\xi)^{-D} > t \}\,dt =\\
      & = 1 + \int_{1}^{\omega^{\abs{g}}} \mu(B_{\bG}(\hat{g}, t^{-1/D})) \,dt \asymp \\ 
      & \asymp 1 +
      \int_{1}^{\omega^{\abs{g}}} t^{-1}\,dt \asymp 1+\abs{g}.
    \end{split}
  \end{equation}
  The second part follows by combining this estimate with~\eqref{eq:Pg-def-estimate} and~\eqref{eq:omega-hat-g-xi}.
\end{proof}

Now we prove the crucial result, stating that the averages of the
functions $\tP_g$ over $A_R$ are uniformly bounded in the $L^\infty$
norm. Later on, the problem of uniform boundedness of weighted
averages of suitably normalized operators $\pi(g)$ will be reduced to
this estimate.

\begin{proposition}
  \label{prop:poisson-cone-avg-bounded}
  The estimate 
  \begin{equation}
   \sum_{g\in A_R}
    \tP_g(\eta) \lasymp \omega^R
  \end{equation}
  holds uniformly in $R$ and $\eta$.
\end{proposition}

\begin{proof}
  By Lemma~\ref{lem:poisson-L1-norm-estimate}, and \eqref{eq:omega-hat-g-xi} we have
  \begin{equation}
    \sum_{g\in A_R}
    \tP_g(\eta) \asymp
    \frac{1}{(1+R)} \sum_{g\in
      A_R} \min\{\omega^{R}, d_\epsilon(\hat{g},\eta)^{-D}\}.
  \end{equation}
  The sum on the right can be estimated in a similar fashion as in the
  proof of Lemma~\ref{lem:poisson-L1-norm-estimate}, yielding
  \begin{equation}
    \begin{split}
      \sum_{g\in A_R} \min\{\omega^{R},
      d_\epsilon(\hat{g},\eta)^{-D}\} & \lasymp \omega^R+\int_1^{\omega^R} \abs{\{ g\in A_R :
       d_\epsilon(\hat{g},\eta)^{-D} > t\}}\,dt \leq \\
     & \leq \omega^R + \int_1^{\omega^R}\abs{ C_R(\eta, \log_{\omega} t)}\,dt
    \end{split}
  \end{equation}
  But for $1 < t < \omega^R$ we have $0< \log_{\omega} t < R$, so
  we may apply Lemma~\ref{lem:growth-of-cones} to obtain
  \begin{equation}
    \begin{split}
       \int_1^{\omega^R}\abs{ C_R(\eta, \log_{\omega} t)}\,dt \asymp \int_1^{\omega^R}\omega^Rt^{-1}\,dt =
      \omega^RR\log\omega,
    \end{split}
  \end{equation}
  which ends the proof.
\end{proof}

\subsection{Approximation on the space of Lipschitz functions}
\label{sec:appr-space-lipsch}

Denote by $\Lip(\bG)$ the vector space of Lipschitz functions on
$(\bG,d_\epsilon)$. Let $\lambda(\phi)$ be the Lipschitz constant of
$\phi\in\Lip(\bG)$. By the Lebesgue differentiation theorem
\cite[Theorem 1.8]{Heinonen2001}, which is valid in particular for any Ahlfors
regular metric measure space, the characteristic functions of balls
span a dense subspace of $L^2(\bG,\mu)$, and since they can be
approximated by Lipschitz functions, it follows that $\Lip(\bG)$ is a
dense subspace of $L^2(\bG,\mu)$.

Define the normalized operator $\wt\pi(g)=\pi(g)/\lVert P_g^{1/2}\rVert_1$. Since
\begin{equation}
  \lVert P_g^{1/2}\rVert_1
  =\langle \pi(g)\one,\one\rangle = \langle\one,\pi(g^{-1})\one\rangle
  = \lVert P_{g^{-1}}^{1/2}\rVert_1,
\end{equation}
where $\one$ is the constant function taking value $1$ everywhere, the operators $\wt\pi(g)$ satisfy $\wt\pi(g)^* =
\wt\pi(g^{-1})$. Moreover, as the next lemma shows, it turns out that
on Lipschitz functions $\wt\pi(g)$ can be approximated by suitable evaluations.

\begin{lemma}
  \label{lem:Lipschitz-eval-approx}
  For $\phi,\psi\in\Lip(\bG)$ we have
  \begin{equation}
    \abs{\langle\wt\pi(g)\phi, \psi\rangle -
      \phi(\check{g})\conj{\psi(\hat{g})}} \lasymp
    \frac{\lambda(\phi)\norm{\psi}_\infty + \lambda(\psi)\norm{\phi}_\infty}{(1+\abs{g})^{1/D}},
  \end{equation}
  uniformly in $g$, $\phi$, and $\psi$.
\end{lemma}

\begin{proof}
  We have
  \begin{equation}
    \begin{split}
      \Big\lvert\langle\wt\pi(g)\phi, \psi\rangle & -
      \phi(\check{g})\conj{\psi(\hat{g})}\Big\rvert \leq \\
      & \leq \abs{ \langle
        \wt\pi(g)\phi,\psi-\psi(\hat{g})\one\rangle} +
      \abs{\conj{\psi(\hat{g})} \Big(\langle\phi,\wt\pi(g^{-1})\one\rangle -
        \phi(\check{g})\Big) } \leq \\
      & \leq \norm{\phi}_\infty \int_\bG \tP_g(\xi) \Big\lvert
      \psi(\xi)-\psi(\hat{g})\Big\rvert\,d\mu(\xi) + \\
      & + \norm{\psi}_\infty \int_\bG \tP_{g^{-1}}(\xi) \Big\lvert
      \phi(\xi)-\phi(\check{g})\Big\rvert \, d\mu(\xi).
    \end{split}
  \end{equation}
  Both terms are of the same form, so we will estimate only the first
  one. Since $\psi$ is Lipschitz, we get
  \begin{equation}
    \int_\bG \tP_g(\xi) \Big\lvert
    \psi(\xi)-\psi(\hat{g})\Big\rvert\,d\mu(\xi)\leq
    \lambda(\psi)\int_{\bG}\tP_g(\xi)d_\epsilon(\hat{g},\xi)\,d\mu(\xi),
  \end{equation}
  and by  integrating
  separately on some ball $B=B_{\bG}(\hat{g},\rho)$ and its complement, we obtain
  \begin{equation}
    \int_B \tP_g(\xi) d_\epsilon(\hat{g},\xi)\,d\mu(\xi) \leq \rho
    \lVert\tP_g\rVert_1 =\rho
  \end{equation}
  and, using Lemma~\ref{lem:poisson-L1-norm-estimate},
  \begin{equation}
    \int_{\bG\setminus B}\tP_g(\xi) d_\epsilon(\hat{g},\xi)\,d\mu(\xi)
    \lasymp \int_{\bG\setminus B}
    \frac{d_\epsilon(\hat{g},\xi)^{1-D}}{1+\abs{g}}\,d\mu(\xi)\lasymp\frac{\rho^{1-D}}{1+\abs{g}},
  \end{equation}
  since $D>1$. We finish by taking $\rho=(1+\abs{g})^{-1/D}$.
\end{proof}

\subsection{Constructing operators in the positive cone}
\label{sec:constr-oper-posit}

For a function $K\in L^\infty(\bG^2,\mu^2)$, define the operator
$T_K\in\mathcal{B}(L^2(\bG,\mu))$ with kernel $K$ by
\begin{equation}
  \langle T_K \phi,\psi\rangle = \int_{\bG^2}
  \phi(\xi)\conj{\psi(\eta)}K(\xi,\eta)\, d\mu^2(\xi,\eta).
\end{equation}

\begin{proposition}
  \label{prop:measure-nuclear-operator-in-cone}
  The operator $T_K$ with kernel $K \geq 0$ is in the positive cone of $\pi$.
\end{proposition}

\begin{proof}
  We will construct a one-parameter family of operators $S_R$, with
  $R\geq0$, in the positive cone of $\pi$, converging to $T_K$ in the
  weak operator topology. We start by fixing $R$ and taking a measurable partition
  $\mathcal{V} = \{V_g : g\in A_R\}$ of $\bG^2$ such that $V_g
  \subseteq \Sigma_2(g)$, which can be obtained for instance by
  putting a linear order on the finite set $A_R$ and taking
  \begin{equation}
    V_g = \Sigma_2(g) \setminus \bigcup_{h<g}\Sigma_2(h).
  \end{equation}
  By Proposition~\ref{prop:double-shadows-cover}, $\mathcal{V}$ indeed covers
  $\bG^2$, and is disjoint by definition. Now, put
  \begin{equation}
    w_g = \int_{V_g} K\,d\mu^2 
  \end{equation}
  and define
  \begin{equation}
    S_R = \sum_{g\in A_R} w_g\wt\pi(g).
  \end{equation}
  For any $\phi,\psi\in\Lip(\bG)$ we get, using
  Lemma~\ref{lem:Lipschitz-eval-approx} and the fact that $V_g\subseteq\Sigma_2(g)$
  \begin{equation}
    \begin{split}
      \Big\lvert\langle S_R\phi,\psi\rangle - & \langle T_K\phi,\psi\rangle\Big\rvert
      \leq \sum_{g\in A_R} \int_{V_g}K\,d\mu^2 \Big\lvert{
        \langle\wt\pi(g)\phi,\psi\rangle -
        \phi(\check{g})\conj{\psi(\hat{g})}}\Big\rvert + \\
      & + \sum_{g\in A_R} \int_{V_g}
      \Big\lvert\phi(\check{g})\conj{\psi(\hat{g})} -
      \phi(\xi)\conj{\psi(\eta)}\Big\rvert K(\xi,\eta)\,d\mu^2(\xi,\eta) \lasymp\\
      & \lasymp\norm{K}_1\frac{\lambda(\phi)\norm{\psi}_\infty +
        \lambda(\psi)\norm{\phi}_\infty}{(1+R)^{1/D}} + \\
      & + \int_{\bG^2} e^{-\epsilon R/2} \Big( \lambda(\phi)\norm{\psi}_\infty +
      \lambda(\psi)\norm{\phi}_\infty\Big)K(\xi,\eta)\,d\mu^2(\xi,\eta),
    \end{split}
  \end{equation}
  so $\langle S_R\phi,\psi\rangle \xrightarrow[R\to\infty]{} \langle
  T_K\phi,\psi\rangle$. 

  By density of $\Lip(\bG)$ it now remains to show that the operators
  $S_R$ are uniformly bounded. First, observe that $w_g \leq
  \norm{K}_\infty \mu(\Sigma_2(g)) \lasymp \omega^{-\abs{g}}$
  uniformly in $g$, and thus by Proposition~\ref{prop:poisson-cone-avg-bounded}
  \begin{equation}
    \norm{S_R\one}_\infty = \sup_{\xi\in\bG}\sum_{g\in A_R} w_g
    \tP_g(\xi) \lasymp 1
  \end{equation}
  uniformly in $R$. Moreover, since $A_R$ is symmetric, the same
  estimate holds for the adjoint operator $S_R^*$. Now, let
  $\phi,\psi\in L^4(\bG,\mu)\subseteq L^2(\bG,\mu)$. The system of
  weights $\{w_g\}$ can be treated as a measure on $A_R$; using the Schwarz
  inequality in the space $L^2(\bG\times A_R,\mu\otimes w)$, we obtain
  \begin{equation}
    \begin{split}
      \lvert\langle & S_R  \phi, \psi \rangle\rvert^2 = \abs{\,
        \int\limits_{\bG\times A_R} \tP_g(\xi)\phi(g^{-1}\xi)\psi(\xi)\,d\xi dg }^2 \leq
      \\
      & \leq
      \int\limits_{\bG\times A_R} \tP_g(\xi)\abs{\phi(g^{-1}\xi)}^2\,d\xi dg 
      \int\limits_{\bG\times A_R} \tP_g(\xi)\abs{\psi(\xi)}^2\,d\xi  dg = \\
      & = \langle S_R\abs{\phi}^2, \one\rangle\langle S_R\one,
      \abs{\psi}^2\rangle \leq\lVert{\phi^2}\rVert_1 \norm{S_R^*\one}_\infty
      \norm{S_R\one}_\infty \lVert{\psi^2}\rVert_1 \lasymp
      \norm{\phi}_2^2\norm{\psi}_2^2
    \end{split}
  \end{equation}
  uniformly in $R$, and thus $S_R$ are uniformly bounded and converge
  to $T_K$ in the weak operator topology as desired.
\end{proof}

By approximating the projections \(P_{E}\colon L^{2}(\bd\Gamma,\mu)\to L^{2}(E,\mu)\) by operators \(T_{K}\) as above, we get the following important corollary.

\begin{corollary}
  \label{lem:projections-in-cone}
  For any measurable set $E\subseteq \bG$ the orthogonal projection
  $P_E$ onto $L^2(E,\mu)\subseteq L^2(\bG,\mu)$ is contained in the
  positive cone of $\pi$.
\end{corollary}

\begin{proof}
  Fix $E\subseteq\bG$. For $\rho > 0$  define $K_\rho\in
  L^\infty(\bG^2,\mu^2)$ as
  \begin{equation}
    K_\rho(\xi,\eta) = \frac{1}{\mu(B_\bG(\xi,\rho))} \chi_E(\xi)\chi_{B_\bG(\xi,\rho)}(\eta)
  \end{equation}
  and let $T_\rho$ be the operator with kernel $K_\rho$. By
  Proposition~\ref{prop:measure-nuclear-operator-in-cone} it is
  contained in the weak operator closure of the cone spanned by
  $\pi(\Gamma)$. For $\phi,\psi\in\Lip(\bG)$ we have 
  \begin{equation}
    \begin{split}
      \Big\lvert \langle T_\rho \phi, \psi \rangle & - \langle
      P_E\phi,\psi\rangle \Big\rvert  \leq \\
      & \leq \bigg\lvert\int_E  \frac{\phi(\xi)}{\mu(B_\bG(\xi,\rho))}
      \int_{B_\bG(\xi,\rho)} \!\!\!\Big(
      \conj{\psi(\eta)}-\conj{\psi(\xi)}\Big)\,d\mu(\eta)d\mu(\xi)\bigg\rvert
      \leq \\
      & \leq \mu(E)\norm{\phi}_\infty \lambda(\psi)\rho
      \xrightarrow[\rho\to 0]{} 0,
  \end{split}
  \end{equation}
  so to finish the proof it only remains to show that the family of operators $T_\rho$ is
  uniformly bounded. We may estimate their operator norms using the
  Schwarz inequality, obtaining
  \begin{equation}
    \norm{T_\rho} \leq \norm{ \int_\bG
      K_\rho(\xi,\eta)\,d\mu(\xi)}_\infty^{1/2} \norm{ \int_\bG K_\rho(\xi,\eta)\,d\mu(\eta)}_\infty^{1/2}.
  \end{equation}
  By Ahlfors regularity we get
  \begin{equation}
    \int_\bG K_\rho(\xi,\eta)\,d\mu(\xi) = 
    \int_E \frac{\chi_{B_\bG(\xi,\rho)}(\eta)}{\mu(B_\bG(\xi,\rho))}\,d\mu(\xi)
    \asymp 
    \int_E \frac{\chi_{B_\bG(\eta,\rho)}(\xi)}{\mu(B_\bG(\eta,\rho))}\,d\mu(\xi)
    \leq 1,
  \end{equation}
  while the second integral is simply equal to $\chi_E(\xi)\leq 1$.
\end{proof}

\section{The boundary representations}
\label{sec:properties}
In this section we show that the boundary representations are
irreducible and weakly contained in the regular representation.

\subsection{Irreducibility}
\label{sec:irreducibility}

We will prove irreducibility by using the following standard observation.
\begin{lemma} 
  \label{lem:irreducibility-criterion}
  Let $\sigma$ be a unitary representation of a group $G$ on a Hilbert
  space $\mathcal{H}$. If there exists a cyclic vector
  $\phi\in\mathcal{H}$ such that the orthogonal projection $P_\phi$
  onto the subspace $\CC\phi$ is contained in the von Neumann algebra
  generated by $\sigma(G)\subseteq \mathcal{B}(\mathcal{H})$, then the
  representation $\sigma$ is irreducible.
\end{lemma}

\begin{proof}
  Let $\mathcal{H}_0 \leq \mathcal{H}$ be a closed nonzero invariant
  subspace. We may take $\psi\in\mathcal{H}_0$ with
  $\langle\phi,\psi\rangle \ne 0$, for otherwise
  \begin{equation}
    \langle \sigma(g)\phi,\psi\rangle =
    \langle\phi,\sigma(g^{-1})\psi\rangle = 0
  \end{equation}
  for all $\psi\in\mathcal{H}_0$ and $g\in G$, which by cyclicity of
  $\phi$  yields $\mathcal{H}_0=0$. Since $P_\phi$ is in the von
  Neumann algebra generated by $\sigma(G)$, the nonzero vector $P_\phi\psi=\lambda\phi$
  belongs to  $\mathcal{H}_0$. Hence, $\mathcal{H}_0$ contains a
  cyclic vector of $\sigma$ and equals $\mathcal{H}$.
\end{proof}

Proving the irreducibility of the boundary representations is now a mere formality.

\begin{theorem}
  \label{thm:irreducibility}
  For any metric $d\in\mathcal{D}(\Gamma)$ the associated boundary
  representation is irreducible.
\end{theorem}

\begin{proof}
  For a positive function $\phi\in L^\infty(\bG)$ the kernel
  $(\xi,\eta)\mapsto \phi(\eta)$ yields a one-dimensional operator $T$
  given by $T\psi = \langle \psi, \one\rangle \phi$. By
  Proposition~\ref{prop:measure-nuclear-operator-in-cone} it is
  contained in the von Neumann algebra of $\pi$, and $T\one = \phi$ is
  contained in the weakly closed span of $\pi(\Gamma)\one$. But
  $L^\infty(\bG)$ is dense in $L^2(\bG)$, and weakly closed subspaces
  are closed, so $\one$ is a cyclic vector of $\pi$. For $\phi=\one$
  the operator $T$ is the orthogonal projection onto $\CC\one$, so by
  Lemma~\ref{lem:irreducibility-criterion} we are done.
\end{proof}

\subsection{Weak containment in the regular representation}
\label{sec:weak-cont-regul}

An important property of any unitary representation is its weak
containment in the regular representation. In case of quasi-regular
representations we may apply the criterion of \cite{Kuhn1994}, which
states that if the action of a discrete group $G$ on a standard Borel
space $(X,\nu)$ is amenable, then all representations obtained by
twisting the quasi-regular representation~\eqref{eq:def-boundary-rep}
of $G$ on $L^2(X,\nu)$ with a cocycle---in our case trivial---are
weakly contained in the regular representation.

Amenability of the action on the Gromov boundary was established in
\cite[Theorem 5.1]{Adams1994}. Namely, for any finite Borel measure
$\mu$ on $\bG$, quasi-invariant under the action of $\Gamma$, the
action of $\Gamma$ on $(\bG,\mu)$ is amenable. We thus obtain the
following.
\begin{proposition}
  \label{prop:weak-containment}
  The boundary representations of $\Gamma$ are weakly contained in the
  regular representation.
\end{proposition}

\section{Classification}
\label{sec:classification}
Here, we investigate unitary equivalence between representations
arising from different metrics on $\Gamma$. We show that the
corresponding boundary representations are equivalent if and only if
the metrics are roughly similar.

For the entire section, let $d'\in\mathcal{D}(\Gamma)$ be another
metric with associated boundary representation $\pi'$. Prime will be
used to indicate objects associated to $d'$, analogous to the ones
defined for $d$.

\subsection{Equivalence in terms of measurable structures}
\label{sec:equiv-terms-meas}

By an \emph{isomorphism} of measure spaces $(X,\mu_X)$ and $(Y,\mu_Y)$
we will understand a Lebesgue isomorphism, i.e.\ a Borel map $F\colon X \to Y$, for which there
exist two subsets $N\subseteq X$ and $N'\subseteq Y$ of measure $0$,
such that the restriction $F\colon X\setminus N \to Y\setminus N'$ is
a Borel isomorphism, and the push-forward $F_*\mu_X$ is equivalent to
$\mu_Y$. Such an isomorphism will be called equivariant if the
corresponding equivariance condition is satisfied almost everywhere.

\begin{lemma}
  \label{lem:equivalent-reps-isomorphic-boundaries}
  Suppose that the representations $\pi$ and $\pi'$ are
  unitarily equivalent. Then there exists a $\Gamma$-equivariant
  isomorphism $F\colon (\bd\Gamma,\mu)\to(\bd\Gamma,\mu')$.
\end{lemma}

\begin{proof}
  Suppose $T\colon L^2(\bd\Gamma,\mu)\to
  L^2(\bd\Gamma,\mu')$ is a unitary intertwining operator. It
  induces a $\Gamma$-equivariant isomorphism 
  \begin{equation}
    \hat{T}\colon
    \mathcal{B}(L^2(\bd\Gamma,\mu')) \to
    \mathcal{B}(L^2(\bd\Gamma,\mu)),
  \end{equation}
  of von Neumann algebras endowed with the conjugation actions of the
  corresponding representations, given by $\hat{T}(S)=T^*ST$. The
  isomorphism $\hat T$ maps the positive cone of $\pi'$ onto the
  positive cone of $\pi$, and preserves orthogonal projections.

  Now, observe that the subalgebra
  $L^\infty(\bd\Gamma,\mu')\leq\mathcal{B}(L^2(\bd\Gamma,\mu'))$ of
  multiplication operators is generated by the orthogonal projections
  $P_E$ onto $L^2(E,\mu')$. They can be characterized as the
  orthogonal projections $P$ such that both $P$ and $I-P$ are in the
  positive cone of $\pi'$. Indeed, one inclusion is a consequence of
  Corollary~\ref{lem:projections-in-cone}. For the other one observe that
  if both $P$ and $I-P$ are in the positive cone, then they preserve
  the natural partial order on functions in
  $L^2(\bd\Gamma,\mu')$. Since the only possible decompositions
  $\one = P\one + (I-P)\one$ into a sum of two orthogonal positive
  functions are of the form $\one = \chi_E+\chi_{E^c}$, for bounded
  positive $\phi$ we get
  \begin{equation}
    P\phi \leq P(\norm{\phi}_\infty\one) \leq \norm{\phi}_\infty \chi_E
  \end{equation}
  for some fixed $E\subseteq \bG$, so $P$ sends $L^2(\bG,\mu')$ into
  $L^2(E,\mu')$. Similarly, the image of $I-P$ is contained in
  $L^2(E^c,\mu')$, so $P=P_E$.

  It follows that $\hat{T}$ restricts to a $\Gamma$-equivariant
  isomorphism between the algebras $L^\infty(\bd\Gamma,\mu')$ and
  $L^\infty(\bd\Gamma,\mu)$. By \cite[Theorem 4, p.\ 238]{Dixmier1981},
  any such isomorphism is induced by an isomorphism
  $F\colon(\bG,\mu)\to(\bG,\mu')$, which is uniquely determined up to
  perturbations on sets of measure $0$. Thus, $\Gamma$-equivariance of
  $\hat{T}$ implies that $F$ is also $\Gamma$-equivariant.
\end{proof}

\subsection{Equivalence in terms of metric structures}
\label{sec:equiv-terms-metr}

We will begin by recalling the following fact, which is standard at least in the context of groups acting on hyperbolic manifolds.

\begin{lemma}
  \label{lem:almost-invariant-measure-on-square}
  The measure class of $\mu^2$ contains a $\Gamma$-invariant
  measure $\nu$ on $\bG^2$ satisfying
  \begin{equation}
    \label{eq:inv-measure-def}
    \nu(E)\asymp \int_E d_\epsilon(\xi,\eta)^{-2D}\,d\mu^2(\xi,\eta).
  \end{equation}
  If the action of \(\Gamma\) is doubly ergodic, then \(\nu\) is unique up to scaling.
\end{lemma}

\begin{proof}
  Let $\tilde\nu = d_\epsilon(\xi,\eta)^{-2D}d\mu^2(\xi,\eta)$ be the
  measure on the right-hand side of \eqref{eq:inv-measure-def}. We
  have $\mu^2$-a.e.
  \begin{equation}
    \begin{split}
      \log_\omega\bigg(\frac{dg_*\tilde\nu}{d\tilde\nu}(\xi,&\eta)\bigg) = \log_\omega
      \bigg(\bigg(
      \frac{d_\epsilon(g^{-1}\xi,g^{-1}\eta)}{d_\epsilon(\xi,\eta)}
      \bigg)^{-2D}
      P_g(\xi)P_g(\eta)\bigg) \approx\\
      & \approx
      2(g^{-1}\xi,g^{-1}\eta)-2(\xi,\eta)+2(g,\xi)+2(g,\eta)-2\abs{g}\approx 0,
  \end{split}
  \end{equation}
  where the last estimate is obtained by expanding the definition of
  the Gromov product, after replacing $\xi$ and $\eta$ by sequences
  $x_n\to\xi$ and $y_n\to\eta$. It follows that the Radon-Nikodym
  derivatives $dg_*\tilde\nu/d\tilde\nu$ are uniformly bounded. It is
  a classical result that in such a situation $\tilde\nu$ can be
  replaced by an equivalent $\Gamma$-invariant measure
  $\nu=\rho\tilde\nu$. One just has to solve the equation
  \begin{equation}\label{eq:cohom-eq}
    \rho(\xi,\eta) = \frac{dg_*\tilde\nu}{d\tilde\nu}(\xi,\eta)\rho( g^{-1}\xi,g^{-1}\eta).
  \end{equation}
  The function
  \begin{equation}
    \rho= \sup_{g\in\Gamma}   \frac{ dg_*\tilde\nu}{d\tilde\nu}(\xi,\eta)
  \end{equation}
  is bounded away from $0$ and $\infty$, and can be seen to satisfy
  equation~\eqref{eq:cohom-eq} by applying supremum over $h$ to the
  cocycle identity
  \begin{equation}
    \frac{d(gh)_*\tilde\nu}{d\tilde\nu}(\xi,\eta) = 
    \frac{dg_*\tilde\nu}{d\tilde\nu}(\xi,\eta) 
    \frac{dh_*\tilde\nu}{d\tilde\nu}(g^{-1}\xi,g^{-1}\eta). 
  \end{equation}
  We thus obtain an invariant measure satisfying
  \eqref{eq:inv-measure-def}. 

  If the action of \(\Gamma\) is doubly ergodic, and $\nu'$ is another
  invariant measure on \(\bd\Gamma^{2}\), equivalent to \(\mu^{2}\),
  then $d\nu'/d\nu$ is a \(\Gamma\)-invariant function, and must be constant, so
  \(\nu'\) is proportional to \(\nu\).
\end{proof}

Recall that in a metric space $(X,d)$ the
cross-ratio of a quadruple of distinct points $x,y,z,w\in X$
is defined as
\begin{equation}
  [x,y,z,w] = \frac{d(x,z)d(y,w)}{d(x,w)d(y,z)}.
\end{equation}

The next lemma is an adaptation of a classical argument from ergodic
theory (see e.g.\ the proof of \cite[Theorem 6.2]{Furman2007}); we
include the detailed proof for the sake of self-containment.

\begin{lemma}
  \label{lem:doubly-ergodic-implies-quasi-mobius}
  If $\pi$ and $\pi'$ are unitarily equivalent, then we have
  \begin{equation}
    \label{eq:double-ergodic-implies-quasi-mobius-statement}
    d'_\epsilon(\xi,\eta)\asymp d_\epsilon(\xi,\eta)^{D/D'}.
  \end{equation}
\end{lemma}

\begin{proof}
  It follows from
  Lemma~\ref{lem:equivalent-reps-isomorphic-boundaries} that there
  exists a $\Gamma$-equivariant isomorphism $F\colon
  (\bd\Gamma,\mu)\to(\bd\Gamma,\mu')$.  Let $\nu$ and $\nu'$ be the
  invariant measures on $\bd\Gamma^2$, corresponding to the metrics
  $d$ and $d'$, constructed in
  Lemma~\ref{lem:almost-invariant-measure-on-square}. Then the
  push-forward $F^2_*\nu$ is an invariant measure on $\bd\Gamma^2$,
  equivalent to $\mu'^2$, and thus proportional to $\nu'$ by double
  ergodicity of \(\mu\) (Theorem~\ref{prop:dbl-erg}). It follows that a.e.
  \begin{equation}
   d'_\epsilon(\xi,\eta)^{-2D'} \asymp 
    d_\epsilon(F^{-1}(\xi),F^{-1}(\eta))^{-2D}
    \frac{d F_*\mu}{d\mu'}(\xi) 
    \frac{dF_*\mu}{d\mu'}(\eta)
  \end{equation}
  After raising this to the power $-1/2D'$ and plugging into the
  definition of the cross-ratio, the Radon-Nikodym derivatives of
  $\mu$ cancel out, and we obtain that the estimate
  \begin{equation}
    \label{eq:F-cross-ratio-estimate}
    [F(\xi_1),F(\xi_2),F(\eta_1),F(\eta_2)]'\asymp [\xi_1,\xi_2,\eta_1,\eta_2]^{D/D'}
  \end{equation}
  holds on a subset  $E\subseteq\bG^4$ of full measure $1$.

  If we take $\rho>0$ and $(\xi_2,\eta_2)$ such that the corresponding
  section
  \begin{equation}
    \{(\xi_1,\eta_1) : (\xi_1,\xi_2,\eta_1,\eta_2) \in E\}
  \end{equation}
  of $E$ has measure $1$, we get
  \begin{equation}
    \label{eq:visual-metrics-equivalence-holder-ae}
    d'_\epsilon(F(\xi_1),F(\eta_1)) \lasymp_{\xi_2,\eta_2,\rho}
      d_\epsilon(\xi_1,\eta_1)^{D/D'} 
  \end{equation}
  for $(\xi_1,\eta_1)$ in a subset of full measure in
  $B_\bG(\eta_2,\rho)^c\times B_\bG(\xi_2,\rho)^c$. By the Fubini's
  theorem, $(\xi_2,\eta_2)$ can be chosen from a set of measure $1$,
  which is in particular dense. We may thus cover $\bG^2$ by finitely
  many sets of the form $B_\bG(\eta_2,\rho)^c\times
  B_\bG(\xi_2,\rho)^c$, and the
  estimate~(\ref{eq:visual-metrics-equivalence-holder-ae}) actually
  holds uniformly on a subset $E'\subseteq\bG^2$ of full measure. 

  Now, denote by $E''$ the set consisting of $\xi\in\bG$ such that the
  section $E'_\xi=\{\eta : (\xi,\eta)\in E'\}$ has full measure. For
  $\xi,\eta\in E''$ and $\zeta\in E'_\xi\cap E'_\eta$ we have
  \begin{equation}
    \begin{split}
      d'_\epsilon(F(\xi),F(\eta)) & \leq
      d'_\epsilon(F(\xi),F(\zeta))+d'_\epsilon(F(\eta),F(\zeta)) \\
      &\lasymp
      d_\epsilon(\xi,\zeta)^{D/D'}+d_\epsilon(\eta,\zeta)^{D/D'}.
    \end{split}
  \end{equation}
  If we let $\zeta$  converge to $\eta$ from within the dense set
  $E'_\xi\cap E'_\eta$, we get a H\"older estimate 
  \begin{equation}
    \label{eq:visual-metrics-equivalence-holder}
    d'_\epsilon(F(\xi),F(\eta)) \lasymp d_\epsilon(\xi,\eta)^{D/D'}
  \end{equation}
  for $F$, satisfied on the subset $E''\subseteq\bG$ of full
  measure. This implies that $F$ is equal a.e.\ to a continuous map
  $H: \bG\to\bG$. By symmetry, from $F^{-1}$ we may construct a
  continuous inverse of $H$, so it is a homeomorphism. Equivariance is
  clear, and by the remark in the last paragraph of Section \ref{sec:gromov-boundary}, $H$
  is in fact the identity
  map. Estimate~(\ref{eq:double-ergodic-implies-quasi-mobius-statement})
  follows from~(\ref{eq:visual-metrics-equivalence-holder}) and
  symmetry of $d_\epsilon$ and $d'_\epsilon$.
\end{proof}

Now we are ready to state and prove the equivalence result.

\begin{theorem}
  \label{thm:inequivalence}
  Let $d,d'\in\mathcal{D}(\Gamma)$ give rise to Patterson-Sullivan
  measures $\mu$ and $\mu'$, and boundary
  representations $\pi$ and $\pi'$.  The following conditions are
  equivalent 
  \begin{enumerate}
  \item $d$ and $d'$ are roughly similar,
  \item $\mu$ and $\mu'$ are equivalent,
  \item $\pi$ and $\pi'$ are unitarily equivalent.
  \end{enumerate}
\end{theorem}

\begin{proof}
  For the metrics $d$ and $d'$, being roughly similar means exactly
  that $d\approx Ad'$. If this is satisfied, the visual metrics
  $d_\epsilon$ and $d'_{A\epsilon}$ are bi-Lipschitz
  equivalent. Hence, the corresponding Hausdorff measures are
  equivalent, and the boundary representations are equivalent by the
  discussion in Section~\ref{sec:bound-repr}.

  For the implication from (3) to (1), we use
  Lemma~\ref{lem:doubly-ergodic-implies-quasi-mobius} to get the
  estimate \eqref{eq:double-ergodic-implies-quasi-mobius-statement} on
  the visual metrics, which implies that the Hausdorff measures $\mu$
  and $\mu'$ are equivalent with Radon-Nikodym derivatives bounded
  away from $0$ and $\infty$. Together with their
  $\Gamma$-quasi-conformality with respect to the corresponding
  metrics on $\Gamma$, this yields the estimate
  \begin{equation}
    \omega^{2(g,\xi)-\abs{g}} \asymp \omega'^{2(g,\xi)'-\abs{g}'}
  \end{equation}
  uniformly in $g$ and $\xi$. By taking the logarithms of both sides,
  and then suprema over $\xi$, using Lemma~\ref{lem:near-geodesic-ray}
  we obtain
  \begin{equation}
    \abs{g} \approx \frac{\log\omega'}{\log\omega}\abs{g}',
  \end{equation}
  which ends the proof.
\end{proof}

\begin{remark}
  \label{rem:pseudometrics}
  It might be tempting to try to extend the class of boundary representations
  even further, by allowing $d$ to be a pseudo-metric. For instance, if
  $\Gamma$ acts properly and cocompactly on a space $X$ then the orbit
  map in general induces a pseudo-metric on $\Gamma$. However,
  pseudo-metrics  do not lead to any new
  representations. In fact, any such pseudo-metric $d$ is roughly isometric
  to a metric 
  \begin{equation}
    d^+(g,h)=\begin{cases}
      d(g,h)+1 & \text{for $g\ne h$}\\
      0 & \text{for $g=h$}
    \end{cases}
  \end{equation}
  yielding the same boundary representation.
\end{remark}

\section{Examples}
\label{sec:applications}

In this section we apply the obtained results to some classes of
groups appearing in nature. It is less self-contained than the
preceding ones---for more details the reader is referred to the
appropriate literature.

\subsection{Fundamental groups of negatively curved manifolds}
\label{sec:cat-1-groups}

The following setting was studied by Bader and Muchnik, who
established irreducibility of boundary representations of fundamental
groups of negatively curved manifolds in \cite{Bader2011}. Let $M$ be
a closed Riemannian manifold with strictly negative curvature. Its
universal cover $\wt M$ is then a hyperbolic metric space on which
$\Gamma=\pi_1(M)$ acts freely and cocompactly by isometries. This
action thus extends to an action of $\Gamma$ on the Gromov boundary
$\bd\wt M$, and yields a unitary representation defined by
formula~\eqref{eq:def-boundary-rep}.  Any orbit map $\Gamma\to \wt M$
induces a hyperbolic metric $d\in\mathcal{D}(\Gamma)$, and since we
may identify $\bG$ with $\bd \wt M$, the resulting representation is
actually the boundary representation of $\Gamma$ associated to the
metric $d$.

The group $\Gamma$ may appear as the fundamental group of
many non-iso\-metric Riemannian manifolds, and any such realization
leads to a potentially different representation. The main theorems of
\cite{Bader2011} state that all these representations are irreducible,
and that they are equivalent if and only if the marked length spectra
of the corresponding manifolds are proportional. By the marked length
spectrum of $M$ we understand the function
$\ell\colon \pi_1(M)\to (0,\infty)$, which to every
$[\gamma]\in\pi_1(M)$ assigns the length of the unique geodesic loop
freely homotopic to $\gamma$.

The irreducibility of the Bader-Muchnik representations is a special
case of Theorem~\ref{thm:irreducibility}. To conclude the
equivalence condition of \cite{Bader2011} using
Theorem~\ref{thm:inequivalence}, it is enough to observe that
proportionality of the marked length spectra is equivalent to rough
similarity of the induced metrics on $\pi_1(M)$. In one direction it
is trivial---the length of the shortest geodesic loop in the free
homotopy class of $g\in\pi_1(M)$ is the translation length of $g$
acting on $\wt M$, and can be expressed in terms of the metric as
\begin{equation}
  \ell(g) = \lim_{n\to\infty} \frac{\lvert g^n\rvert}{n}.
\end{equation}
For the other direction we may resort to \cite[Theorem 2.2]{Otal1992},
which states that the marked length spectrum determines the
cross-ratio on the boundary. Hence, proportional marked length spectra
lead to cross-ratios satisfying $[\cdot]'=[\cdot]^\alpha$, and thus
they arise from roughly similar metrics, by the arguments from the
proofs of Lemma \ref{lem:doubly-ergodic-implies-quasi-mobius} and
Theorem~\ref{thm:inequivalence}.

\subsection{Green metrics and Poisson boundaries}
\label{sec:green-metr-poiss}

Let $\Gamma$ be a non-elementary hyperbolic group, and let $\nu$ be a
symmetric probability measure on $\Gamma$, whose support generates
$\Gamma$. Such a measure gives rise to a random walk on
$\Gamma$; denote by $F(g,h)$ the probability that starting at $g$, it
ever reaches $h$. Assume that $\nu$ has \emph{exponential moment},
i.e. there exists $\lambda > 0$ such that
\begin{equation}
\label{eq:exp-moment}
  \sum_{g\in\Gamma} e^{\lambda \abs{g}}\nu(g) < \infty,
\end{equation}
where the length is taken with respect to any word metric on $\Gamma$, 
and that for any $r$ there exists a constant $C(r)$ for which
\begin{equation}
\label{eq:F-geodesic}
  F(x,y) \leq C(r)F(x,v)F(v,y)
\end{equation}
for $v$ within distance $r$ from a geodesic (again, with respect to
some fixed word metric) joining $x$ and $y$. These two conditions hold
for any finitely supported measure \cite[Corollary 1.2]{Blachere2011}, and
ensure that the \emph{Green metric}
\begin{equation}
d_G(g,h)=-\log F(g,h),
\end{equation}
which was introduced in \cite{Blachere2006}, and studied further in
\cite{Blachere2008}, belongs to the class $\mathcal{D}(\Gamma)$.

Trajectories of the random walk with law $\nu$ almost surely converge
to a point in $\bG$, and the hitting probability defines the
\emph{harmonic measure} $\hat{\nu}$ on $\bG$ associated with $\nu$. By
\cite[Theorems 1.1(ii) and 1.5]{Blachere2011}, it turns out that
$\hat\nu$ is equivalent to the Patterson-Sullivan measure associated
with the Green metric $d_G$, and thus yields the same quasi-regular
representation. On the other hand, a measure with exponential moment
has finite first moment, and by \cite[Theorem 7.4]{Kaimanovich2000},
the Gromov boundary with the harmonic measure $(\bG,\hat{\nu})$ is
actually isomorphic to the Poisson boundary of $(\Gamma,\nu)$. By
Theorem~\ref{thm:irreducibility}, we therefore obtain the following
new result.
\begin{theorem}
\label{thm:harmonic-irreducibility}
Let $\nu$ be a symmetric measure on $\Gamma$, satisfying
conditions~\eqref{eq:exp-moment} and \eqref{eq:F-geodesic}. Then the
quasi-regular representation associated with the Poisson boundary
$(\bG,\hat\nu)$ is irreducible.
\end{theorem}

\appendix

\section{Double ergodicity of Patterson-Sullivan measures}
\label{sec:double-ergod-patt-1}

Although in many contexts double ergodicity of Patterson-Sullivan measures is known, there is no proof in the literature, which would apply in the general context of this paper. We have discussed this problem with Uri Bader and Alex Furman, who are planning to include a general proof of a stronger property, called \emph{isometric double ergodicity} in their forthcoming paper \cite{Bader2014}. In this appendix, we give a detailed proof of double ergodicity of Patterson-Sullivan measures, which is required for the classification of boundary representations. It is based on the ideas explained to us by Bader and Furman.

\subsection{Measurable geodesic flow}
\label{sec:meas-geod-flow}

Let \(\Gamma\) be a hyperbolic group endowed with a metric \(d\in\mathcal{D}(\Gamma)\), giving rise to the Patterson-Sullivan measure \(\mu\). For \(x=(\xi,\eta)\in\bd\Gamma^{2}\) we will write \(x_{-}=\xi\), and \(x_{+}=\eta\). Define for \(g\in\Gamma\), and \(x \in\bd\Gamma^{2}\)
\begin{equation}
  \label{eq:cocycle-def}
  c_{\pm}(g,x) = \log_{\omega} \frac{dg^{-1}_{*}\mu}{d\mu}(x_{\pm}) \approx 2(g^{-1},x_{\pm})-\abs{g},
\end{equation}
and put
\begin{equation}
  c(g,x) = \frac{c_{+}(g,x)-c_{-}(g,x)}{2} \approx (g^{-1},x_{+})-(g^{-1},x_{-}).
\end{equation}
Then \(c_{\pm}\), and therefore also \(c\), satisfy  the cocycle identity
\begin{equation}
  c_{*}(gh,x) = c_{*}(h,x) + c_{*}(g,hx)
\end{equation}
for all \(g,h\in\Gamma\) and \(x\in\Omega\), where \(\Omega\subseteq\bd\Gamma^{2}\setminus\{ (\xi,\eta) : \xi=\eta\}\) has full measure. Thanks to removing the diagonal, every \(x\in\Omega\) constitutes a pair of endpoints of some roughly geodesic ray.

\begin{lemma}\label{prop:cocycles-geodesics-estimates}
  The following estimates hold uniformly for \(x\in\Omega\),  \(r\in\RR\), and any rough geodesic \(\gamma\) from \(x_{-}\) to \(x_{+}\).
  \begin{enumerate}
  \item \(c_{+}(\gamma(r)^{-1},x)\approx r + c_{+}(\gamma(0)^{-1},x)\),
  \item \(c_{-}(\gamma(r)^{-1},x)\approx -r + c_{-}(\gamma(0)^{-1},x)\),
  \item \(c(\gamma(r)^{-1},x)\approx r + c(\gamma(0)^{-1},x)\).
  \end{enumerate}
\end{lemma}

\begin{proof}
  For (1) and (2) observe that
  \begin{equation}
    c_{\pm}(\gamma_{}(r)^{-1},x) \approx 2(\gamma_{}(r),x_{\pm})-\abs{\gamma_{}(r)} \approx \liminf_{s\to\pm \infty} (\abs{\gamma_{}(s)} - \abs{s-r}),
  \end{equation}
  and for fixed \(r\) the difference \(s-r\) is either eventually positive or eventually negative, depending on whether \(s\to\infty\) or \(s\to -\infty\), yielding
  \begin{equation}
    c_{+}(\gamma_{}(r)^{-1},x) \approx r + \liminf_{s\to\infty} (\abs{\gamma_{}(s)}-s) \approx r + c_{+}(\gamma(0)^{-1},x),
  \end{equation}
  and
  \begin{equation}
    c_{-}(\gamma_{}(r)^{-1},x) \approx -r + \liminf_{s\to-\infty} (\abs{\gamma_{}(s)}+s)\approx -r + c_{-}(\gamma(0)^{-1},x).
  \end{equation}
  Estimate (3) now follows immediately.
\end{proof}

Fix an invariant measure \(\nu\) on \(\bd\Gamma^{2}\) according to Lemma~\ref{lem:almost-invariant-measure-on-square}, and 
define an action of \(\Gamma\) on \(\Omega\times \RR\) by
\begin{equation}
  g(x,t) = (gx, t-c(g,x)).
\end{equation}
It commutes with the standard translation action of \(\RR\), which we will denote by \((x,t)+s \mapsto (x,t+s)\), and leaves invariant the product measure \(\nu\otimes dt\), where \(dt\) stands for the Lebesgue measure on \(\RR\).

\subsection{The fundamental domain}
\label{sec:fundamental-domain}

We will now construct a fundamental domain for the action of \(\Gamma\) on \(\Omega\times\RR\). By this we mean a strict fundamental domain, i.e.\ a Borel set \(\Delta\subseteq \Omega\times\RR\), whose \(\Gamma\)-translates are disjoint and cover \(\Omega\times\RR\).  First, for \(h,\theta>0\) define 
\begin{equation}
  D_{\theta,h} = \{ (x,t)\in\Omega\times(-h,h) : d_\epsilon(x_{-},x_{+}) >
  \theta\}.
\end{equation}
A subset of \(\Omega\times\RR\) will be called \emph{bounded}, if it is contained is some \(D_{\theta,h}\). Bounded sets have finite measure, since
\begin{equation}
  \nu\otimes dt (D_{\theta,h}) \asymp 2h \int_{d_{\epsilon}(x_{-},x_{+})>\theta} d_{\epsilon}(x_{-},x_{+})^{-2D} d\mu^{2}(x) \leq 2h\theta^{-2D}.
\end{equation}
Furthermore, the property of being bounded is invariant under the action of \(\Gamma\). Indeed, \(\Gamma\) acts on \(\bd\Gamma\) by bi-Lipschitz homeomorphisms, and \(\abs{c(g,x)}\aless \abs{g} \), so the image of \(D_{\theta,h}\) under any \(g\in\Gamma\) remains bounded.

\begin{lemma} \label{prop:translates-cover}
  There exist \(\theta,h>0\), such that the set \(D_{\theta,h}\) meets every \(\Gamma\)-orbit in \(\Omega\times\RR\).
\end{lemma}

\begin{proof}
  Let \((x,t)\in\Omega\times\RR\), and choose a rough geodesic \(\gamma\) from \(x_{-}\) to \(x_{+}\). Put \(r=t-c(\gamma(0)^{-1}, x)\), and \(g=\gamma(r)^{-1}\). Then by Lemma~\ref{prop:cocycles-geodesics-estimates} we have \(c(g,x)\approx t\), and so
  \begin{equation}
    g(x,t)=(gx,t-c(g,x)) \in \Omega\times(-h,h)
  \end{equation}
  with some \(h\) depending only on the constant in this uniform estimate. Moreover,
  \begin{equation}
    (gx_{-},gx_{+})\approx (x_{-},x_{+})_{g^{-1}}\approx 0,
  \end{equation}
  since \(g^{-1}=\gamma_{x}(r)\) lies on a rough geodesic between \(x_{-}\) and \(x_{+}\). Hence,
  \begin{equation}
    d_{\epsilon}(gx_{-},gx_{+}) \asymp e^{-\epsilon(gx_{-},gx_{+})}\asymp 1,
  \end{equation}
  so for some \(\theta>0\) depending on the estimate constants, we have \(g(x,t)\in D_{\theta,h}\).
\end{proof}

\begin{lemma}\label{prop:fund-domain-properness}
  For any \(\theta,h>0\) the set of \(g\in\Gamma\) such that \(gD_{\theta,h}\cap D_{\theta,h}\ne\emptyset\) is finite.
\end{lemma}

\begin{proof}
Fix \(\theta\) and \(h\), and pick \(g\in\Gamma\) such that \(gD_{\theta,h}\cap D_{\theta,h}\ne\emptyset\). This means that there exists \((x,t)\in D_{\theta,h}\) such that \(g(x,t)\in D_{\theta,h}\), and the following three estimates are satisfied:
\begin{equation}
    (g^{-1},x_{+}) \approx_{h} (g^{-1},x_{-}),\qquad (x_{-},x_{+}) \approx_{\theta} 0,\qquad
  (gx_{-},gx_{+}) \approx_{\theta} 0.
\end{equation}
Additionally, by \eqref{eq:gromov-product-inverse}, the first estimate implies that
\begin{equation}
  (g,gx_{+})\approx_{h} (g,gx_{-}).
\end{equation}
This yields
\begin{equation}
(g^{-1},x_{+})\approx_{h} \min\{ (g^{-1},x_{+}),(g^{-1},x_{-})\}\aless (x_{-},x_{+})\approx_{\theta} 0,
\end{equation}
and similarly
\begin{equation}
  (g,gx_{+}) \approx_{h} \min\{ (g,gx_{+}),(g,gx_{-})\}\aless (gx_{-},gx_{+}) \approx_{\theta} 0.
\end{equation}
By \eqref{eq:gromov-product-inverse} we thus have
\begin{equation}
  \abs{g}\approx (g^{-1},x_{+}) + (g,gx_{+}) \approx_{\theta,h}0,
\end{equation}
so the elements \(g\in \Gamma\) satisfying our conditions are all contained in a ball whose radius depends only on \(h\) and \(\theta\).
\end{proof}

\begin{lemma}\label{prop:strict-fund-domain}
  The action of \(\Gamma\) on \(\Omega\times\RR\) admits a bounded  Borel fundamental domain.
\end{lemma}

\begin{proof}
  This is a standard argument from descriptive set theory. We will use
  \cite[Theorem 12.16]{Kechris1995}, stating that any
  partition of a Polish space into closed subsets such that the
  saturation of any open set (i.e.\ the union of parts intersecting the set) is Borel, admits a Borel transversal.

  By Lemma~\ref{prop:translates-cover} we may choose \(\theta,h>0\) such that the \(\Gamma\)-translates of \(D_{\theta,h}\) cover \(\Omega\times\RR\). The set
  \begin{equation}
    X = \{ (x,t)\in\bd\Gamma^{2}\times(-h,h) : d_{\epsilon}(x_{-},x_{+})>\theta\}
  \end{equation}
  is clearly a Polish space, containing \(D_{\theta,h}\). Now,
  partition \(X\) into orbits of \(\Gamma\) in \(D_{\theta,h}\), and
  singletons in \(X\setminus D_{\theta,h}\). The parts of this
  partition are finite by Lemma~\ref{prop:fund-domain-properness}, and
  hence closed. Moreover, the saturation of any open set
  $U\subseteq X$ is
  \begin{equation}
    (U \setminus D_{\theta,h}) \cup \bigcup_{g\in \Gamma} (gU\cap
    D_{\theta,h}),
  \end{equation}
  and is a Borel set. Hence, the partition admits a Borel transversal, and its intersection with \(D_{\theta,h}\) is a fundamental domain for the action of \(\Gamma\) on \(\Omega\times \RR\). 
\end{proof}

Let \(\Delta\subseteq \Omega\times\RR\) be a bounded fundamental domain for the action of \(\Gamma\), whose existence is asserted by Lemma~\ref{prop:strict-fund-domain}. For any \(p\in\Omega\times\RR\) and \(t\in\RR\) there exists a unique element \(\kappa_{\Delta}(p,t)\in\Gamma\) such that
\begin{equation}
  p+t\in\kappa_{\Delta}(p,t)^{-1}\Delta,
\end{equation}
and the map \(\kappa_{\Delta}\colon \Omega\times\RR\times\RR \to \Gamma\) is clearly Borel.

Now, we may identify the quotient space \((\Omega\times\RR)/\Gamma\) with \(\Delta\), and since \(\Omega\times\RR\) is endowed with the translation action of \(\RR\), commuting with the action of \(\Gamma\), this action descends to a measure-preserving flow \(\Phi_{s}^{\Delta}\) on \(\Delta\), given by
\begin{equation}
  \Phi_{s}^{\Delta}(p) = \kappa_{\Delta}(p,s)(p+s).
\end{equation}
In particular, on \(\Delta\times\RR\) the cocycle identity
\begin{equation}
  \kappa(\Phi_{t}^{\Delta}(p),s)\kappa(p,t) = \kappa(p,s+t)
\end{equation}
holds.

\begin{lemma} \label{prop:cocycle-is-quasigeo}
  For any bounded fundamental domain \(\Delta\) of the action of \(\Gamma\) on \(\Omega\), and any \(q=(x,t)\in \Delta\), the mapping \(\RR_{+}\ni s\mapsto \kappa_{\Delta}(q,s)^{-1}\) is a quasi-geodesic ray in \(\Gamma\) with endpoint \(x_{+}\).
\end{lemma}

\begin{proof}
  Assume that \(\Delta\subseteq D_{\theta,h}\), and take \(s>0\). As a consequence of the cocycle identity for \(\kappa_{\Delta}\), we
  get
  \begin{equation}\label{eq:cocycle-quasigeo-dist-1}
    d( \kappa(q,t+s)^{-1}, \kappa(q,t)^{-1}) = \lvert{ \kappa(\Phi_{t}^{\Delta}(q),s)}\rvert,
  \end{equation}
  so we need to estimate the length of \(\kappa_{\Delta}(p,s)\) in terms of \(s\) for arbitrary \(p\in\Delta\).

  For \(s< h\) we have \(p+s \in D_{\theta,2h}\), and simultaneously
  \begin{equation}
     \kappa(p,s)(p+s) = \Phi_{s}^{\Delta}(p) \in \Delta \subseteq D_{\theta,2h},
  \end{equation}
  so  by Lemma \ref{prop:fund-domain-properness}, the possible values of \(\kappa(p,s)\) belong to some ball in \(\Gamma\), depending only on \(\theta\) and \(h\). This yields a uniform upper bound for \(\lvert\kappa_{\Delta}(p,s)\rvert\)
  valid for \(\abs{s}<h\). Therefore, 
  \begin{equation}
    d( \kappa(q,t+s)^{-1}, \kappa(q,t)^{-1}) \lasymp \lceil{s/h}\rceil\lasymp 1 + s
  \end{equation}
  for arbitrary \(s\).
  
  Now, put \(p=(x,t)\), and observe that
  \begin{equation}
    \Phi_{s}^{\Delta}(p)=\kappa_{\Delta}(p,s)(p+s) = (\kappa_{\Delta}(p,s)x, t+s-c(\kappa_{\Delta}(p,s),x))
  \end{equation}
  is in \(\Delta\subseteq D_{\theta,h}\), so
  \begin{equation}
    c(\kappa_{\Delta}(p,s),x)\approx_{h} s.
  \end{equation}
  On the other hand, we have
  \begin{multline}
    c(\kappa_{\Delta}(p,s),x)\approx (\kappa_{\Delta}(p,s)^{-1},x_{+}) - (\kappa_{\Delta}(p,s)^{-1},x_{-}) \leq \\
    \leq (\kappa_{\Delta}(p,s)^{-1},x_{+}) \leq \lvert \kappa_{\Delta}(p,s)^{-1} \rvert,
  \end{multline}
  yielding a lower bound in \eqref{eq:cocycle-quasigeo-dist-1}, and proving that indeed we are dealing with a quasi-geodesic ray. Also, we see above that \((\kappa_{\Delta}(p,s)^{-1},x_{+}) \agtr_{h} s\), so the endpoint of our ray is \(x_{+}\).
\end{proof}

\subsection{Lebesgue differentiation}
\label{sec:lebesg-diff}

Recall that by the Lebesgue Differentiation Theorem \cite[Theorem 1.8]{Heinonen2001}, for any positive \(f\in L^{1}(\bd\Gamma,\mu)\) we have
\begin{equation}
  \lim_{r\to 0} \frac{1}{\mu(B(\xi,r))} \int_{B(\xi,r)} f\,d\mu = f(\xi)
\end{equation}
a.e.\ on \(\bd\Gamma\). In this section we will introduce a variant of this theorem, where one integrates \(f\) against \(g_{*}\mu\) instead of the normalized restriction of \(\mu\) to a ball. First, recall that for \(f\in L^{1}(\bd\Gamma,\mu)\) the maximal function \(Mf\) is defined as
\begin{equation}
  Mf(\xi) = \sup_{B\ni\xi} \frac{1}{\mu(B)}\int_{B}\abs{f}\,d\mu,
\end{equation}
where the supremum is taken over all balls in \(\bd\Gamma\) containing \(\xi\). By \cite[Theorem 2.2]{Heinonen2001} it satisfies the weak \(L^{1}\) estimate
\begin{equation}
  \mu\{ \xi\in\bd\Gamma : Mf(\xi)>t\} \lasymp \frac{1}{t}\norm{f}_{1}.
\end{equation}
Let us now define  \(N_{\rho}f\) by
\begin{equation}
  N_{\rho}f(\xi) = \sup_{\Sigma(g,\rho)\ni \xi}  \int \abs{f}\,dg_{*}\mu. 
\end{equation}

\begin{lemma} \label{prop:weak-l1-for-maximal-fn}
  The functions \(N_{\rho}f\) and \(Mf\) satisfy the estimate
  \begin{equation}
    N_{\rho}f(\xi) \lasymp_{\rho} Mf(\xi)
  \end{equation}
  for \(\xi\in \bd\Gamma\). In particular, \(N_{\rho}f\) satisfies the weak \(L^{1}\) estimate
  \begin{equation}
    \mu\{ \xi\in\bd\Gamma : N_{\rho}f(\xi)>t\} \lasymp_{\rho} \frac{1}{t}\norm{f}_{1}.
\end{equation}
\end{lemma}

\begin{proof}
  Without loss of generality we may assume that \(f\) is positive. Let \(\eta\in\bd\Gamma\). It suffices to show that for \(g\in\Gamma\) such that \(\eta\in\Sigma(g,\rho)\) we have 
  \begin{equation}\label{eq:maximal-estimate-proof-1}
    \int f\,dg_{*}\mu \lasymp_{\rho} Mf(\eta)
  \end{equation}
  uniformly in \(g\).
  By \eqref{eq:omega-hat-g-xi} we have
  \begin{multline} \label{eq:rn-derivative-estimate-in-maximal-fn-proof}
    \frac{dg_{*}\mu}{d\mu}(\xi) \asymp \min\{
    \omega^{\abs{g}},\omega^{-\abs{g}}d_{\epsilon}(\hat{g},\xi)^{-2D}\} =
    \\ =
    \begin{cases} \omega^{\abs{g}}, & d_{\epsilon}(\hat{g},\xi) <
      e^{-\epsilon\abs{g}}\\
      \omega^{-\abs{g}}d_{\epsilon}(\hat{g},\xi)^{-2D}&\text{otherwise,}
    \end{cases}
  \end{multline}
  and for \(e^{-\epsilon\abs{g}} < d_{\epsilon}(\hat{g}, \xi) <  e^{-\epsilon(\abs{g} - \rho)}\) we have \(\omega^{-\abs{g}}d_{\epsilon}(\hat{g},\xi)^{-2D} \asymp_{\rho} \omega^{\abs{g}}\), so finally
  \begin{equation}
    \int f\,dg_{*}\mu \asymp_{\rho} \int f\phi_{g}\,d\mu,
  \end{equation}
  where
  \begin{equation}
    \phi_{g}(\xi) =
    \begin{cases}
      \omega^{\abs{g}}& \xi\in\Sigma(g,\rho),\\
      \omega^{-\abs{g}}d_{\epsilon}(\hat{g},\xi)^{-2D} & \xi\not\in \Sigma(g,\rho).
    \end{cases}
  \end{equation}
  The function \(\phi_{g}\) is radial with respect to \(\hat{g}\), constant on \(\Sigma(g,\rho)\), and non-increasing in \(d_{\epsilon}(\hat{g},\xi)\). It can be therefore approximated from above by simple functions of the form
  \begin{equation}
    \psi_{g} = \sum_{i=1}^{n} \alpha_{i}\chi_{B(\hat{g},r_{i})}
  \end{equation}
  with \(\norm{\psi_{g}-\phi_{g}}_{1}\) arbitrarily small, and all \(B(\hat{g},r_{i})\) containing \(\Sigma(g,\rho)\). Now,
  \begin{equation}
    \int f\phi_{g} \,d\mu \leq \int f\psi_{g}\,d\mu = \sum_{i=1}^{n} \alpha_{i} \mu(B(\hat{g},r_{i})) \frac{1}{\mu(B(\hat{g},r_{i}))} \int_{B(\hat{g},r_{i})} f\,d\mu,
  \end{equation}
  and since \(\eta\in\Sigma(g,\rho)\subseteq B(\hat{g},r_{i})\) for all \(i\),
  \begin{equation}
    \int f\phi_{g}\,d\mu \leq \sum_{i=1}^{n}\alpha_{i}\mu(B(\hat{g},r_{i})) Mf(\eta) = Mf(\eta) \norm{\psi_{g}}_{1}.
  \end{equation}
  Since \(\psi_{g}\) can be chosen arbitrarily close to \(\phi_{g}\) in \(L^{1}\)-norm, and \(\phi_{g}\asymp dg_{*}\mu/d\mu\),  in the last estimate there is a uniform bound on \(\norm{\psi_{g}}_{1}\), yielding \eqref{eq:maximal-estimate-proof-1}. 

  The second statement follows from the corresponding estimate on \(Mf\).
\end{proof}

\begin{lemma}\label{prop:weak-convergence-to-dirac}
  Let \(\xi\in\bd\Gamma\). In the weak* topology the push-forwards \(g_{*}\mu\) converge to the unit mass supported at \(\xi\) as \(g\to\xi\).
\end{lemma}

\begin{proof}
  Fix \(\xi\in\bd\Gamma\), and \( f\in C(\bd\Gamma)\). For any \(\alpha>0\) there exists \(\rho>0\) such that for \(\eta\in B=B_{\bd\Gamma}(\xi,\rho)\) we have \(\abs{f(\xi)-f(\eta)} < \alpha\), and
  \begin{equation}
    \begin{split}
      \abs{\int f\,dg_{*}\mu - f(\xi)} & \leq \int_{B}\abs{f - f(\xi)}\,dg_{*}\mu + \int_{B^{c}} \abs{f-f(\xi)}\,dg_{*}\mu \leq \\
      & \leq \alpha  + 2\norm{f}_{\infty} \int_{B^{c}} \frac{dg_{*}\mu}{d\mu}\,d\mu.
    \end{split}
  \end{equation}
  But if \(\abs{g}\) is sufficiently large, then by~\eqref{eq:rn-derivative-estimate-in-maximal-fn-proof}, on \(B^{c}\) we may write
  \begin{equation}
    \int_{B^{c}} \frac{dg_{*}\mu}{d\mu}(\xi)\,d\mu(\xi) \asymp \int_{B^{c}} \omega^{-\abs{g}} d_{\epsilon}(\hat{g},\xi)^{-2D}\,d\mu(\xi) \leq \omega^{-\abs{g}} \rho^{-2D},
  \end{equation}
  and this converges to \(0\) as \(\abs{g}\to\infty\).
\end{proof}

For a function \(f\colon \Gamma\to\CC\), and a point \(\xi\in\bd\Gamma\), we will say that a number \(z\in\CC\) is \emph{the radial limit of \(f\) at \(\xi\)}, denoted
\begin{equation}
  z = \radlim_{g\to\xi} f(g),
\end{equation}
if for any quasi-geodesic ray \(\gamma\colon \RR_{+}\to\Gamma\) with endpoint \(\xi\), we have
\begin{equation}
  z =\lim_{t\to\infty} f(\gamma(t)).
\end{equation}
For \(L\geq 1\), \(C\geq 0\), and \(\xi\in\bd\Gamma\) denote by \(Q_{L,C}(\xi)\) the set of all \((L,C)\)-quasi-geodesic rays in \(\Gamma\) with endpoint \(\xi\). 

\begin{lemma} \label{prop:shadow-convergence-for-quasigeodesics}
  For every \(L\geq 1\) and \(C\geq 0\) there exists \(\rho>0\) such that for every \(\xi\in\bd\Gamma\), and every \(\gamma\in Q_{L,C}(\xi)\), for sufficiently large \(t\) we have \(\xi\in\Sigma(\gamma(t),\rho)\).
\end{lemma}

\begin{proof}
  Let \(\gamma\in Q_{L,C}(\xi)\). Then, for some constants \(L',C'\), depending only on \(L,C\) and the metric \(d\in \mathcal{D}(\Gamma)\), which is by definition quasi-isometric to a word metric,  \(\gamma\) is also an \((L',C')\)-quasi-geodesic ray in some fixed Cayley graph \(X\) of \(\Gamma\). By the Morse Lemma \cite[Theorem III.H.1.7]{Bridson1999}, there exists a geodesic ray \(\gamma'\) in \(X\) within Hausdorff distance \(r\) from \(\gamma\), where \(r\) depends only on \(L',C'\), and the metric on \(X\). 

  Now, by \cite[Theorem A.1]{Blachere2011}, the natural quasi-isometry \(\phi\colon X\to(\Gamma,d)\) (i.e.\ the quasi-inverse of the embedding \(\Gamma\to X\), retracting edges to their endpoints), takes geodesics to quasi-rulers, which by \cite[Lemma A.2]{Blachere2011} can be reparametrized to become rough geodesics. In particular, from \(\gamma'\) we obtain a rough geodesic \(\gamma_{1}\) in \((\Gamma,d)\), within Hausdorff distance \(r'\) of \(\gamma\), where \(r'\) depends on \(r\) and the constants of \(\phi\). Ultimately, since \(X\), \(\phi\), and the metric \(d\in\mathcal{D}(\Gamma)\) are fixed, \(r'\) depends only on \(L\) and \(C\).

  Let \(t\in\RR_{+}\), and denote \(g=\gamma(t)\). There exists \(s\in\RR_{+}\) such that \(d(g,\gamma_{1}(s))\leq 2r'\) and \(\abs{g}\approx_{r'} \abs{\gamma_{1}(s)}\). Furthermore, if \(s\) is sufficiently large, which can be guaranteed by taking sufficiently large \(t\), then
  \begin{equation}
    (\gamma_{1}(s),\xi) \approx \abs{\gamma_{1}(s)},
  \end{equation}
  so
  \begin{equation}
    (\hat{g},\xi) \agtr \min\{ (\hat{g},g),(g,\gamma_{1}(s)), (\gamma_{1}(s),\xi)\} \agtr_{r'} \abs{g},
  \end{equation}
  and finally
  \begin{equation}
    d_{\epsilon}(\hat{g},\xi) \lasymp_{r'} e^{-\epsilon\abs{g}}.
  \end{equation}
  The constant in this estimate yields the required \(\rho\), which depends on \(r'\), and hence only on \(L\) and \(C\).
\end{proof}

Now we are ready to prove our variant of the Lebesgue Differentiation Theorem, along the lines of the standard proof using maximal functions \cite[Section 2.7]{Heinonen2001}.

\begin{proposition}\label{prop:radial-lebesgue-diff}
  Let \(f\in L^{1}(\bd\Gamma,\mu)\). Then for \(\xi\in\bd\Gamma\)
  \begin{equation}\label{eq:lebesgue-diff-convergence}
    \radlim_{g\to \xi} \int \abs{f(\eta)-f(\xi)}\,dg_{*}\mu(\eta) = 0
  \end{equation}
  holds a.e.
\end{proposition}

\begin{proof}
  For \(L\geq 1\), and \(C\geq 0\) let \(E_{L,C}\subseteq\bd\Gamma\) be the set of those \(\xi\) for which
  \begin{equation}
    \lim_{t\to\infty} \int \abs{f(\eta)-f(\xi)}\,d\gamma(t)_{*}\mu(\eta)=0
  \end{equation}
  for all \(\gamma\in Q_{L,C}(\xi)\). The set on which the radial convergence~\eqref{eq:lebesgue-diff-convergence} holds is the intersection of all the sets \(E_{L,C}\). But this intersection is the same as the intersection of a countable subfamily with integer \(L\) and \(C\), so it is enough to show that \(E_{L,C}\) has full measure in \(\bd\Gamma\) for every \(L\) and \(C\).

  Now, we will proceed as in the proof of the classical Lebesgue differentiation theorem. Fix \(L\) and \(C\). For \(f\in L^{1}(\bd\Gamma,\mu)\) define the function \(\Lambda f\) by
  \begin{equation}
    \Lambda f(\xi) = \sup_{\gamma\in Q_{L,C}(\xi)} \limsup_{t\to \infty} \int \abs{ f(\eta)-f(\xi)}\,d\gamma(t)_{*}\mu(\eta).
  \end{equation}
  We need to show that it vanishes almost everywhere.

  We have \(\Lambda(f_{1}+f_{2})\leq \Lambda f_{1} + \Lambda f_{2}\), and if \(f\) is continuous, then \(\Lambda f(\xi) = 0\) by Lemma \ref{prop:weak-convergence-to-dirac}. Moreover, if by Lemma~\ref{prop:shadow-convergence-for-quasigeodesics} we take \(\rho\) such that for every \(\gamma\in Q_{L,C}(\xi)\) we have \(\xi\in \Sigma(\gamma(t),\rho)\) for large \(t\), then we claim that the following estimate is satisfied:
  \begin{equation}
    \Lambda f(\xi) \leq N_{\rho}f(\xi) + \abs{f(\xi)}.
  \end{equation}
  Indeed, for \(\gamma\in Q_{L,C}(\xi)\) we have
  \begin{equation}
    \begin{split}
      \limsup_{t\to \infty} & \int \abs{
        f(\eta)-f(\xi)}\,d\gamma(t)_{*}\mu(\eta) \leq\\ & \leq \limsup_{t\to
        \infty} \int \abs{f(\xi)}\,d\gamma(t)_{*}\mu(\eta)  + \limsup_{t\to
        \infty} \int \abs{f(\eta)}\,d\gamma(t)_{*}\mu(\eta),
    \end{split}
  \end{equation}
  where the first summand is just \(\abs{f(\xi)}\), while in the second one, for sufficiently large \(t\) we have \(\xi\in \Sigma(\gamma(t),\rho)\), and consequently
  \begin{equation}
    \int \abs{f(\eta)}\,d\gamma(t)_{*}\mu(\eta) \leq N_{\rho}f(\xi).
  \end{equation}
  Now, we may write \(f=f_{n}+c_{n}\), where \(c_{n}\) is continuous, and \(\norm{f_{n}}_{1}\to 0\), to get
  \begin{equation}
    \Lambda f \leq \Lambda f_{n} + \Lambda c_{n} \leq N_{\rho}f_{n} + \abs{f_{n}}.
  \end{equation}
  But using the weak \(L^{1}\) estimate from Lemma~\ref{prop:weak-l1-for-maximal-fn} we obtain
  \begin{equation}
    \mu\{ \xi : \Lambda f(\xi) > s\}  \lasymp \frac{1}{s}  \norm{f_{n}}_{1} \xrightarrow[n\to\infty]{} 0,
  \end{equation}
  and the function \(\Lambda f\) vanishes almost everywhere.
\end{proof}

\subsection{Double ergodicity}
\label{sec:double-ergodicity}

We are now ready to show double ergodicity of Patterson-Sullivan measures.

\begin{theorem} \label{prop:dbl-erg}
  Let \(\Gamma\) be a hyperbolic group endowed with a metric \(d\in\mathcal{D}(\Gamma)\) giving rise to the Patterson-Sullivan measure \(\mu\). Then the action of \(\Gamma\) on \((\bd\Gamma,\mu)\) is doubly ergodic.
\end{theorem}

\begin{proof}
  Let \(E\subseteq \bd\Gamma^{2}\) be a \(\Gamma\)-invariant subset of positive measure. By taking the appropriate intersection, we may assume that \(E\subseteq \Omega\). For \(\xi\in\bd\Gamma\) denote
  \begin{equation}
    E_{\xi} = \{ \eta\in\bd\Gamma : (\xi,\eta)\in E\}.
  \end{equation}
  If we assume that \(E\) is not of full measure, i.e.
  \begin{equation}
    \mu^{2}(E) = \int_{\bd\Gamma} \mu(E_{\xi})\,d\mu(\xi) < 1,
  \end{equation}
  then there exists a subset \(A\subseteq \bd\Gamma\) of positive
  measure, and \(\alpha \in (0,1)\), such that for \(\xi\in A\) we
  have \(\mu(E_{\xi}) < \alpha\). Furthermore, by Proposition~\ref{prop:radial-lebesgue-diff}, the set
  \begin{equation}
    \wt A = \{(\xi,\eta)\in E : \xi\in A,\; \radlim_{g\to \eta} \mu(g^{-1}E_{\xi})=1\}
  \end{equation}
  has positive measure.

  Now, choose a bounded fundamental domain \(\Delta\) for the action of \(\Gamma\) on \(\Omega\times\RR\), such that
  \begin{equation}
    B = \Delta\cap ( \wt A\times\RR)
  \end{equation}
  has positive measure. By Lemma~\ref{prop:cocycle-is-quasigeo}, for any \(p=(\xi,\eta,t)\in B\), the map \(s\mapsto \kappa_{\Delta}(p,s)^{-1}\) is a quasi-geodesic ray with endpoint \(\eta\), and we have
  \begin{equation}
    \lim_{s\to\infty} \mu(\kappa(p,s) E_{\xi}) = \lim_{s\to\infty} \int_{E_{\xi}}d\kappa(p,s)^{-1}_{*}\mu= \one_{E_{\xi}}(\eta) = 1.
  \end{equation}
  We may therefore fix \(T>0\) such that the set
  \begin{equation}
    C = \{p=(\xi,\eta,t)\in B : \mu(\kappa(p,s)E_{\xi}) > \alpha\;\text{for \( s > T\)}\}
  \end{equation}
  has positive measure.   By the Poincar\'e Recurrence Theorem applied to the flow \(\Phi^{\Delta}_{s}\), there exists \(s>T\) such that \(\Phi^{\Delta}_{-s}(C)\cap C\ne\emptyset\). We can partition this set into countably many subsets indexed by elements of \(\Gamma\),
  \begin{equation}
    C_{g} = \{ p\in \Phi^{\Delta}_{-s}(C)\cap C : \kappa(p,s)=g\},
  \end{equation}
  and there exists \(g\) such that \(C_{g}\ne\emptyset\). Hence, also its image \(\pr_{1}(C_{g})\) under the projection \(\pr_{1}\colon \bd\Gamma\times\bd\Gamma\times\RR\to\bd\Gamma\) onto the first axis is non-empty. But for \(\xi\in \pr_{1}(C_{g})\) there exist \(\eta\) and \(t\) such that \(p=(\xi,\eta,t)\in C_{g}\subseteq C\), and since we have chosen \(s>T\), by definition of \(C\) we have
  \begin{equation}
    \mu(E_{g\xi}) = \mu(gE_{\xi}) = \mu(\kappa(p,s)E_{\xi}) > \alpha.
  \end{equation}
  To conclude, observe that
  \begin{equation}
     gC_{g} + s = g( C_{g} + s)  = \Phi^{\Delta}_{s}(C_{g}) \subseteq C,
  \end{equation}
  so
  \begin{equation}
    \pr_{1}(gC_{g}) = \pr_{1}(gC_{g}+s) \subseteq \pr_{1}(C) \subseteq \pr_{1}(B) \subseteq \pr_{1}(\wt A\times\RR) = A,
  \end{equation}
  thus for \(\xi\in C_{g}\), by definition of \(A\) we have \(\mu(E_{g\xi}) < \alpha\), contradicting the initial assumption that \(E\) is not of full measure.
\end{proof}

\bibliographystyle{plain}
\bibliography{../../../library}

\begin{thebibliography}{10}

\bibitem{Adams1994}
S.~Adams.
\newblock {Boundary amenability for word hyperbolic groups and an application
  to smooth dynamics of simple groups}.
\newblock {\em Topology}, 33(4):765--783, 1994.

\bibitem{Bader2014}
Uri Bader and Alex Furman.
\newblock {Boundary maps and simplicity of the Lyapunov spectrum}.

\bibitem{Bader2011}
Uri Bader and Roman Muchnik.
\newblock {Boundary unitary representations---irreducibility and rigidity}.
\newblock {\em J. Mod. Dyn.}, 5(1):49--69, 2011.

\bibitem{Bekka2002}
Mohammed E.~B. Bekka and Michael Cowling.
\newblock {Some irreducible unitary representations of $G(K)$ for a simple
  algebraic group $G$ over an algebraic number field $K$}.
\newblock {\em Math. Zeitschrift}, 241(4):731--741, 2002.

\bibitem{Binder1993}
Michael~W. Binder.
\newblock {On Induced Representations of Discrete Groups}.
\newblock {\em Proc. Am. Math. Soc.}, 118(1):301--309, 1993.

\bibitem{Blachere2006}
S{\'{e}}bastien Blach{\`{e}}re and Sara Brofferio.
\newblock {Internal Diffusion Limited Aggregation on Discrete Groups Having
  Exponential Growth}.
\newblock {\em Probab. Theory Relat. Fields}, 137(3--4):323--343, 2006.

\bibitem{Blachere2008}
S{\'{e}}bastien Blach{\`{e}}re, Peter Ha{\"{i}}ssinsky, and Pierre Mathieu.
\newblock {Asymptotic entropy and Green speed for random walks on countable
  groups}.
\newblock {\em Ann. Probab.}, 36(3):1134--1152, 2008.

\bibitem{Blachere2011}
S{\'{e}}bastien Blach{\`{e}}re, Peter Ha{\"{i}}ssinsky, and Pierre Mathieu.
\newblock {Harmonic measures versus quasiconformal measures for hyperbolic
  groups}.
\newblock {\em Ann. Sci. l'{\'{E}}cole Norm. Sup{\'{e}}rieure}, 44(4):683--721,
  2011.

\bibitem{Bonk2000}
M.~Bonk and O.~Schramm.
\newblock {Embeddings of Gromov hyperbolic spaces}.
\newblock {\em Geom. Funct. Anal.}, 10(2):266--306, 2000.

\bibitem{Bridson1999}
Martin~R. Bridson and Andr{\'{e}} Haefliger.
\newblock {\em {Metric Spaces of Non-Positive Curvature}}.
\newblock Springer, 1999.

\bibitem{Burger1997}
Marc Burger and Pierre de~la Harpe.
\newblock {Constructing irreducible representations of discrete groups}.
\newblock {\em Proc. Math. Sci.}, 107(3):223--235, 1997.

\bibitem{Coornaert1993}
Michel Coornaert.
\newblock {Mesures de Patterson-Sullivan sur le bord d'un espace hyperbolique
  au sens de Gromov}.
\newblock {\em Pacific J. Math.}, 159(2):241--270, 1993.

\bibitem{Corwin1975}
Lawrence Corwin.
\newblock {Induced representations of discrete groups}.
\newblock {\em Proc. Am. Math. Soc.}, 47(2):279, 1975.

\bibitem{Cowling1991}
Michael Cowling and Tim Steger.
\newblock {The irreducibility of restrictions of unitary representations of
  lattices}.
\newblock {\em J. f{\"{u}}r die reine und Angew. Math.}, 420:85--98, 1991.

\bibitem{Dixmier1981}
Jacques Dixmier.
\newblock {\em {Von Neumann Algebras}}.
\newblock North-Holland Pub. Co., 1981.

\bibitem{Dudko2015}
Artem Dudko.
\newblock {On irreducibility of Koopman representations of Higman-Thompson
  groups}.
\newblock {\em arXiv}, 1512.02687, 2015.

\bibitem{Figa-Talamanca1991}
Alessandro Fig{\`{a}}-Talamanca and Claudio Nebbia.
\newblock {\em {Harmonic Analysis and Representation Theory for Groups Acting
  on Homogeneous Trees}}.
\newblock CUP Archive, 1991.

\bibitem{Figa-Talamanca1983}
Alessandro Fig{\`{a}}-Talamanca and Massimo~A. Picardello.
\newblock {\em {Harmonic analysis on free groups}}, volume~87 of {\em Lecture
  Notes in Pure and Applied Mathematics}.
\newblock Marcel Dekker Inc., New York, 1983.

\bibitem{Figa-Talamanca1994}
Alessandro Fig{\`{a}}-Talamanca and Tim Steger.
\newblock {Harmonic analysis for anisotropic random walks on homogeneous
  trees}.
\newblock {\em Mem. Am. Math. Soc.}, 110(531):xii+68, 1994.

\bibitem{Furman2007}
Alex Furman.
\newblock {Measurable rigidity of actions on infinite measure homogeneous
  spaces, II}.
\newblock {\em J. Am. Math. Soc.}, 21(02):479--513, 2007.

\bibitem{Garncarek2012}
{\L}ukasz Garncarek.
\newblock {Analogs of principal series representations for Thompson's groups
  $F$ and $T$}.
\newblock {\em Indiana Univ. Math. J.}, 61(2):619--626, 2012.

\bibitem{Garncarek2010}
{\L}ukasz Garncarek.
\newblock {Irreducibility of some representations of the groups of
  symplectomorphisms and contactomorphisms}.
\newblock {\em Colloq. Math.}, 134(2):287--296, 2014.

\bibitem{Heinonen2001}
Juha Heinonen.
\newblock {\em {Lectures on Analysis on Metric Spaces}}.
\newblock Springer, 2001.

\bibitem{Kaimanovich2000}
Vadim~A. Kaimanovich.
\newblock {The Poisson Formula for Groups with Hyperbolic Properties}.
\newblock {\em Ann. Math.}, 152(3):659--692, 2000.

\bibitem{IlyaKapovich}
Ilya Kapovich and Nadia Benakli.
\newblock {Boundaries of hyperbolic groups}.
\newblock In {\em Comb. Geom. Gr. theory (New York, 2000/Hoboken, NJ, 2001)},
  volume 296 of {\em Contemp. Math.}, pages 39--93. Amer. Math. Soc., 2002.

\bibitem{Kechris1995}
Alexander~S. Kechris.
\newblock {\em {Classical Descriptive Set Theory}}, volume 156 of {\em Graduate
  Texts in Mathematics}.
\newblock Springer New York, New York, NY, 1995.

\bibitem{Kuhn1994}
M.~Gabriella Kuhn.
\newblock {Amenable Actions and Weak Containment of Certain Representations of
  Discrete Groups}.
\newblock {\em Proc. Am. Math. Soc.}, 122(3):751--757, 1994.

\bibitem{Mackey1976}
George~W. MacKey.
\newblock {\em {The Theory of Unitary Group Representations}}.
\newblock The University of Chicago Press, 1976.

\bibitem{Otal1992}
Jean-Pierre Otal.
\newblock {Sur la g{\'{e}}ometrie symplectique de l'espace des
  g{\'{e}}od{\'{e}}siques d'une vari{\'{e}}t{\'{e}} {\`{a}} courbure
  n{\'{e}}gative}.
\newblock {\em Rev. Matem{\'{a}}tica Iberoam.}, 8(3):441--456, 1992.

\bibitem{Rogers1970}
C.~A. Rogers.
\newblock {\em {Hausdorff measures}}.
\newblock Cambridge University Press, 1970.

\bibitem{Vershik1975}
A.~M. Vershik, I.~M. Gel'fand, and M.~I. Graev.
\newblock {Representations of the group of diffeomorphisms}.
\newblock {\em Russ. Math. Surv.}, 30(1):1--50, 1975.

\end{thebibliography}
\end{document}